\documentclass[12pt]{article} 
\RequirePackage{filecontents}        
\usepackage[utf8]{inputenc}
\usepackage{enumitem}
\newcommand{\rk}[1]{\textbf{\color{red}#1}}

\usepackage{geometry} 
\usepackage{authblk}

\usepackage{graphicx} 

\usepackage{float}			
\usepackage{comment}
\usepackage{graphicx} 
\usepackage{caption}
\usepackage{subcaption}
\usepackage{verbatim} 
\usepackage{amsmath}

\usepackage{amssymb}
\usepackage{lmodern} 
\usepackage{mathtools}
\usepackage{bm}
\usepackage{esvect}
\usepackage[colorlinks]{hyperref}
\usepackage{bbm}
\usepackage{dsfont}
\usepackage[T1]{fontenc}
\usepackage{amsfonts}
\usepackage[font=small,labelfont=bf]{caption}

\usepackage{amsthm}
\numberwithin{equation}{section}

\theoremstyle{plain}

\newtheorem{thm}{Theorem}[section]

\newtheorem{lem}[thm]{Lemma}
\newtheorem*{lem*}{Lemma}
\newtheorem{prop}[thm]{Proposition}

\newtheorem{cor}[thm]{Corollary}

\newtheorem{rem}{Remark}[section]
\newtheorem*{rem*}{Remark}

\newcommand\NN{\mathbb{N}} 
\newcommand\ZZ{\mathbb{Z}} 
\newcommand\QQ{\mathbb{Q}} 
\newcommand\RR{\mathbb{R}} 
\newcommand\R{\mathbb{R}} 

\newcommand\EE{\mathbb{E}} 
\newcommand\BB{\mathbb{B}} 
\newcommand\PP{\mathbb{P}} 
 
\newcommand\diffsym{\bigtriangleup} 
\newcommand\restr[2]{{
  \left.\kern-\nulldelimiterspace 
  #1 
  \vphantom{\big|} 
  \right|_{#2} 
  }}

\newcommand\Cc{\mathcal{C}}

\newcommand\compl{\mathsf{c}}

\DeclareMathOperator{\dist}{dist}

\DeclareMathOperator{\Leng}{Length}
\DeclareMathOperator{\Var}{Var}
\DeclareMathOperator{\Vol}{Vol}

\DeclareMathOperator{\Diam}{Diam}
\DeclareMathOperator{\Ind}{Ind}

\DeclareMathOperator{\Cross}{Cross}

\title{Random pseudometrics\\
and applications}

\author{{\bf Vivek Dewan} \\ENS de Lyon \\15 parvis Ren\'e Descartes \\69342 Lyon Cedex 07\\
	{\small\tt vivek.dewan@ens-lyon.fr}
	\and
	{\bf Damien Gayet}\\
	Institut Fourier\\ Universit\'e Grenoble Alpes\\
	100 rue des Maths, 38610 Gi\`eres, France
	{\small\tt damien.gayet@univ-grenoble-alpes.fr}
}

\begin{document}
\maketitle

\begin{abstract}
Let $T$ be a random ergodic pseudometric over $\RR^d$. This setting generalizes the classical \emph{first passage percolation} (FPP) over $\ZZ^d$. We provide simple conditions on $T$, see~(\ref{annuliT}) (decay of instant one-arms) and (\ref{AI}) (exponential quasi-independence), that ensure the positivity of its time constants (Theorem~\ref{positifT}), that is almost surely, the pseudo-distance given by $T$  from the origin is asymptotically  a norm. Combining this general result with previously known ones, we prove that 
\begin{itemize}
	\item the known phase transition for Gaussian percolation in the case of fields with positive correlations with exponentially fast decay
holds for Gaussian FPP (Theorem~\ref{posBF}), including the natural Bargmann-Fock model;
\item the known phase transition for 
Voronoi percolation also extends to the associated FPP (Theorem~\ref{bevo}); 
\item the same happens for Boolean percolation (Corollary~\ref{bouboule}) for radii with exponential tails, a result which was known without this condition.
\item  We prove the positivity of the constant for  random continuous Riemannian metrics (Theorem~\ref{smoothmetric0}), including cases with infinite correlations in dimension $d=2$.
\item Finally, we show that the critical exponent for the one-arm, if exists, is bounded above by $d-1$ (Corollary~\ref{ceg}). This holds for
bond Bernoulli percolation, planar Gaussian fields, planar Voronoi percolation, and Boolean percolation with exponential small tails.
\end{itemize}
\end{abstract}

\tableofcontents

\section{Introduction}
\paragraph{Classical FPP.} First passage percolation (FPP) was first introduced by Hammersley and Welsh in 1965~\cite{Hw}. In its simplest version, it provides a random pseudometric over the graph made of the edges of the hypercubic lattice $\ZZ^d$. For any $p\in [0,1]$, any edge is given independently a number $\sigma_p$, 0 with probability $p$ and $1$ with probability $1-p$. The pseudo-distance beween two extremities of an edge is defined by this number. The pseudo-distance $T(x,y)$ between two vertices is the least pseudo-length of the continuous paths made of edges from $x$ to $y$. 

An important object in this context is the family of \emph{time constants} 
$(\mu_p(v))_{v\in \RR^d}$, that is the limits $(\lim_n \frac{1}n T(0,nv))_{v\in \R^d}$ of the large rescaled pseudo-distances to the origin. The existence of these limit is given by the ergodicity of the model, see Theorem~\ref{fixedFPP}.
It has been proved, see Theorem~\ref{tp}, that the large scale behaviour for $T$ follows the same phase transition as the associated Bernoulli percolation, namely that $\mu_p$ is positive if and only if $p$ is smaller than $p_c(d)$, the critical parameter for Bernoulli percolation in dimension $d$.
Recall that for $p<p_c(d)$, almost surely there is no infinite component of $\{\sigma_p=0\}$, and for $p>p_c(d)$, almost surely there is an infinite component of this set. For FPP, another classical result holds, namely the Cox-Durett ball shape theorem:
for $p<p_c(d)$, the large pseudo-balls of radius $t$ centered on the origin and defined by $T$ are almost surely close to $t$ times a deterministic convex compact with non-empty interior, see Theorem~\ref{ballshape}, whereas for $p\geq p_c(d)$, the pseudo-balls of radius $t$ rescaled by $1/t$ converge to the whole space. 

\paragraph{Random pseudometrics.} In~\cite{ziesche2016first}, a wide generalization of the classical FPP was proposed: general random  ergodic pseudometrics $T: (\RR^d)^2\to \RR_+$ over the whole affine space $\RR^d$.
In this continuous setting we can also define the family of time constants $(\mu(v))_{v\in \R^d}$, under mild conditions, see Theorem~\ref{fixeddirT}. 
 In this paper we  prove a general theorem, see Theorem~\ref{positifT}, which asserts that under two simply stated main conditions, the time constants associated with $T$ are positive. More precisely, if $T$ is ergodic, satisfies an exponential decay of correlations, see~(\ref{AI}), and if the probability that the origin and a large sphere are at vanishing $T$-distance decreases polynomially fast with degree greater than $d-1$, see~(\ref{annuliT}), then the time constants of $T$ are positive. Quite surprinsigly, 
Theorem~\ref{positifT} applies to all the known natural sorts of FPP, discrete or continuous, with the notable exception of the Gaussian free field~\cite{ding2019upper}, where the correlations are too strong for this setting. When $T$ is Lipschitz, which is the case of all the applications, except Riemannian percolation, we also prove a ball shape theorem, see Theorem~\ref{ballshapeBFT}. In the sequel, we present the four main applications. 

\paragraph{Random densities and colourings.} Historically, the first natural generalization of the classical FPP on $\mathbb Z^d$ has been provided by random  measurable \emph{colourings} $\sigma: \RR^d \to \{0,1\}$. Here, the associate pseudo-distance $T(x,y)$ is the least integral of $\sigma$ over the piecewise $C^1$ paths between two points $x,y$ of $\RR^d$, see~(\ref{defmetric}). 
This can be generalized to random \emph{densities}, that is random maps $\sigma: \R^d \to \R_+$. 
In this context, under the two aformentioned main conditions, Theorem~\ref{positifT} applies, see Corollary~\ref{positif}. In the case of colourings,  $T$ is always 1-Lipschitz, so that the ball shape theorem applies, see Corollary~\ref{ballshapeBF}.

\paragraph{Boolean FPP.}
The first colouring model which has been studied seems to be the \emph{Boolean} or \emph{continuous} percolation. Since it appears that the latter adjective covers a far larger class of models, we will refer to this model only as \emph{Boolean}. In this setting, the colouring $\sigma_{\nu,\lambda}$ is the characteristic function of the (complement of the) union of balls of random radii with law $\nu$ centered on random points of a Poisson process of intensity $\lambda$. It is now classical that for a fixed radii law, the percolation model undergoes a phase transition with parameter $\lambda$. Again, the phase transition concerns the infinite components of $\{\sigma_{\nu, \lambda}=0\}.$ Recently, it has been proved that a similar phase transition holds for the associated FPP, see Theorem~\ref{gouthe}. As an application of Theorem~\ref{positifT}, we recover this result in a restrictive situation, namely an exponential tail of the radii law $\nu$, see Corollary~\ref{bouboule}.

\paragraph{Voronoi FPP.}
Another continuous model based  on a Poisson point process over $\RR^d$ is the \emph{Voronoi percolation}. In this setting, the locally finite set $X$ of random points induces a partition of the space into \emph{Voronoi cells} defined by the points which are closest to a particular point in $X$. 
In Voronoi percolation, for a given $p\in [0,1]$, all the points in a given random cell are given a common number $\sigma_p$,  0 or 1, with respective probability $p$ and $1-p$, as in Bernoulli percolation, and this is done independently over the cells.  
It is classical that this model undergoes a phase transition for the infinite components of $\{\sigma_p=0\}$. Recently, new results about the associated percolation and criticity properties have been proved, see Theorems~\ref{BoRi}, \ref{RSWbevo} and~\ref{DRT}. We prove in this paper, using the aforementioned results and Theorem~\ref{positifT}, that a phase transition occurs for the associated FPP, see Theorem~\ref{bevo}.

\paragraph{Gaussian FPP.}
Also very recently, another class of continuous percolation model was reborn, \emph{Gaussian percolation}, that is connectivity properties associated with the sign of a stationnary Gaussian field over $\RR^d$. Common features with Bernoulli percolation have been revealed some years ago for planar fields with positive and strongly decorrelating fields, see Theorems~\ref{RSW_2} and~\ref{uniqcomp}, the latter providing a phase transition for the levels of the random field. More precisely, for $p\in \RR$ and a random real centered Gaussian field $f$  over $\RR^2$, let $\sigma_p$ be the colouring equal to 0 if $f+p\leq 0$ and $1$ if $f+p>0$. Then,  almost surely $\{\sigma_p=0\}$ has an infinite component if and only if $p<0$. In this planar context, for the same conditions on the correlations, we apply  Theorem~\ref{positifT} to prove that the FPP model associated with $\sigma_p$ undergoes the same phase transition, see Theorem~\ref{posBF}. All this applies to the natural Bargmann-Fock model defined by~(\ref{BF}).

\paragraph{Riemannian FPP.}

Another and very different continuous model was introduced in~\cite{lagatta2010shape}. 
In this situation, a random continuous Riemannian metric $g$ is given over $\RR^d$, and the associated pseudometric $T$ is given by the associated random distance. Under some moment conditions and if the model has finite correlations, the author of the aformentioned paper proved that $T$ is comparable to the Euclidean distance, see Theorem~\ref{smoothmetricLW}. We apply Theorem~\ref{positifT} to prove a more general result with weaker conditions, see~Theorem~\ref{smoothmetric}. In particular, in dimension 2 and for metrics associated with strongly decorrelating Gaussian fields, we give examples with infinite correlations, see Corollary~\ref{smoothmetric2}. 

\paragraph{Other models.}
In the realm of Gaussian fields, we can also, instead of integrating the sign of the function, integrate a positive functional of the function, see~(\ref{gauss2metric}). For instance, we can integrate the density $\max (0, f)$ instead of its sign. We prove that this model also undergoes a phase transition with the level $p$, see Theorem~\ref{gauss2}. Theorem~\ref{posBF} becomes in fact a particular case of said theorem.

Another application is the Ising model. In this case and for the range of temperature for which we can say something, the time constant is vanishing, so that it does not use our main Theorem~\ref{positifT}. Consequently, we refer for instance to~\cite{Velenik} for definitions and classical properties. We prove that for high negative temperature (anti-ferromagnetic) and for positive (ferromagnetic) temperature above the critical temperature, the time constant vanishes, where the associated random pseudometric is associated with the random colouring given by the spins, see Theorem~\ref{ising}.

\paragraph{Critical exponent.} As a direct consequence of Theorem~\ref{positifT}, we prove that for a model satisfying condition~(\ref{AI}) (quasi independence) and such that $\mu=0$, 
the probability that there exists an instant path from the origin to a sphere of size $R$ cannot decrease faster than $R^{-(d-1)}$, see Corollary~\ref{cegT}. When the pseudometric is given by the integral of a non-negative function $\sigma$, it implies the same for zero paths. For critical bond percolation over $\ZZ^d$, it is known  that this probability is of order $R^{-2}$ for $d>10$, see~\cite[Theorem 1]{kozma} and~\cite[Theorem 1.6]{fitzner2017mean}. In dimension 2, a consequence~\cite{lawler2002} of Smirnov's result is that in the case of the triangular lattice it is of order $R^{-5/48}$. Corollary~\ref{ceg} implies that the critical exponent for bond percolation, if it exists, is less or equal to  $d-1$, see Corollary~\ref{cebp}. For planar critical Gaussian fields ($p=0$), our result gives that the one-arm probability decreases no faster than $R^{-1}$, see Corollary~\ref{cegaussian}, and the same holds for planar Voronoi critical percolation ($p=1/2$), see Corollary~\ref{ceV}. For critical Boolean percolation in every dimension, the decay cannot be faster than $R^{-(d-1)}$, see Corollary~\ref{ceb}. We also provide a shorter and more general proof due to Hugo Vanneuville of these corollaries, see Theorem~\ref{hugo}.

\paragraph{Open questions}

\begin{itemize}
	\item One main conjecture for discrete FPP is the universality of the fluctuations of $T(0,x)-\mu(x) = o(x)$. It is conjectured~\cite[\S 3.1]{auffinger201750}  that $$\text{var } T(0,x) \sim_{\|x\|\to \infty} \|x\|^{2/3}$$ on $\RR^2$, where the symbol $\sim$ has various interpretations. Does the previous estimate hold for isotropic Gaussian fields, for instance the Bargmann-Fock field?
	Note that in our continuous setting, there are none of the problems caused by the rigidity of the lattice. Moreover, if the field is isotropic, the limit ball is a disk, which should help. However, one of the main problems in our context is the infinite dependency, an issue which does not arise in classical Bernoulli percolation.  
	\item Another conjecture is related to the deviations of the geodesics of the almost metric from the straight line, for instance the maximal distance between these two kinds of geodesics.  It is conjectured that this distance should be of order $\|x\|^\gamma$ for a certain exponent $\gamma<1$, see~\cite[\S 4.2]{auffinger201750}. It is very natural to assume that this should be the case for Gaussian fields. 
	\item The proof of Corollary~\ref{positif} involves a combinatorial bound, which must be fought by, among others, the asymptotic independence given by condition~(\ref{AI}) (asymptotic independence). In the Gaussian case, this independence is provided by the exponentially fast decay of the correlation function. If said function decreases only polynomially, the combinatorics win and we cannot get any upper bound. However, we cannot find any profound, non-technical reason for this need of exponential decay. 
\end{itemize}

\paragraph{Structure of the paper. } In section~\ref{statement}, we present in more details the various FPP models and the  results for general random pseudometrics, densities and colourings. In section~\ref{applications}, we present the various applications of the main results to Gaussian, Voronoi, Boolean and Riemannian percolation. In section~\ref{secproof}, we give the proof of the main general theorems, in particular Theorem~\ref{positifT}. In section~\ref{secapp}, we then explain how they can be applied to our applications. 

\paragraph{Acknowledgements.}  We would like to thank warmly Vincent Beffara, Jean-Baptise Gou\'er\'e and Hugo Vanneuville for corrections, valuable discussions and precious suggestions. 
We also thank R\'egine Marchand for a first discussion on this subject and Rapha\"el Cerf for references. 
The research leading to these results has received funding from the French Agence nationale de la recherche, ANR-15CE40-0007-01.

\section{Statement of the general results}\label{statement}
\subsection{FPP over lattices}\label{class}

 The classical FPP is more general than the Bernoulli one we described in the introduction. We refer to~\cite{auffinger201750} for a an introductive introduction to the subject.
Recall that $\mathbb L^d = (\ZZ^d, \mathbb E^d)$  denotes the hypercubic lattice. Let $\nu$ be a probability law on $\RR_+$. 
Let 
$$\sigma_\nu: \mathbb E^d\to \R^+$$ 
be such that 
every edge $e\in \mathbb E^d$  is endowed with an independent time  $\sigma_\nu(e)\in \R_+$ following the law $\nu$. 
Now, for any two vertices $(x,y)$ in $\ZZ^d$,  a \textit{path} between $x$ and $y$ is a continuous path from $x $ to $y$ made of edges. Then, the random time or pseudo-distance between $x$ and $y$ is defined by:
\begin{equation}\label{defFPPmetric}
T(x,y):=\inf_{\gamma\text{ path }x\rightarrow y}\sum\limits_{e\in \gamma}\sigma_\nu(e).
\end{equation}
We have hence endowed $\ZZ^d$ with a random pseudometric. It is not a metric since $T$ can vanish even if the points are different. Note that in the Bernoulli case explained in the introduction, 
if $p=0$, $T$ is the graph distance, and if $p=1$, $T$ degenerates to 0.
For any probability measure $\nu$ on $\RR_+$, define the following condition:
\begin{enumerate}[series=condi]
	\item (Finite moment)
\begin{equation}\label{numoment}
\EE\min(\sigma_\nu(1)^2,\cdots,\sigma^2_\nu(2d))<\infty
\end{equation}
\end{enumerate}
where the $\sigma_\nu(i)$'s are i.i.d random variables with law $\nu$.
The first main result in this domain is 
a consequence of the ergodicity of the model:
\begin{thm}\cite{Hw} \label{fixedFPP}
Let $\nu$  be a probability measure over $\R^+$ satisfying condition~(\ref{numoment}) (finite moment). Then,
	there exists a $\QQ$-semi-norm $\mu_\nu$ such that
	\begin{equation}\label{convergences}
 \lim\limits_{n \to +\infty} \frac{1}n T(0,nw) = \mu_\nu(w)\ \text{  almost surely and $L^1$}.
	\end{equation}
\end{thm}
Let $p_c(d)$ be the critical threshold for Bernoulli bond percolation on $\ZZ^d$, that is
$$p_c(d)=\sup\{p\in [0,1], \text{ there is no infinite component of } \{\sigma_p=0\} \text{ a.s.}\}. $$
It is well known~\cite{grimmett} that for any $d\geq 2$, $p_c(d)\in ]0,1[$, and that $p_c(2) =1/2$. The second result is the main one. It asserts that laws that don't allow too fast times for an edge, the time constant $\mu_\nu$ is positive, and vice versa:
\begin{thm}\label{tp} \cite{kesten1986aspects} Let $\nu$  be a probability measure over $\R^+$ satisfying condition~(\ref{numoment}) (finite moment). Then, $$  \mu_\nu \text{ is a norm }\Leftrightarrow \PP[\nu =0]< p_c(d).$$
\end{thm}Notice that for Bernoulli percolation, the condition is equivalent
to $p< p_c(d)$.
For subcritical laws, a natural question is to study the geometry of the large balls defined by the pseudometric $T$. For this define:
$$\forall t\geq 0,  \ B_t= 
\{x\in \ZZ^d, T(x,0)\leq t\}$$
the family of balls defined by the pseudometric $T$.
 In 1981, J. T. Cox and R. Durrett proved the following geometric result:
\begin{thm}\label{ballshape}\cite{cox} (for $d=2$) \cite{kesten1986aspects} (for $d\geq 2$) Let $\nu$  be a probability measure over $\R^+$ satisfying condition~(\ref{numoment}) (finite moment) and
$T$ be defined by (\ref{defFPPmetric}). 
\begin{enumerate}
\item If $\PP[\nu=0]\geq p_c(d)$, then for any $M>0$, 
$$\PP [M\mathbb B \subset \frac{1}t B_t \text{ for $t$ large enough } ]= 1,$$
where $\mathbb B$ denotes the unit standard open ball in $\R^d$.
\item If	$\PP[\nu=0]<p_c(d)$, there exists a deterministic compact set $K\subset\RR^d$ with non-empty interior, such that for any positive $\epsilon$,
	\begin{equation}\label{KB}
	\PP\left[(1-\epsilon)K\subset\frac{1}tB_t\subset(1+\epsilon)K\text{ for all $t$ large enough }\right]=1.
	\end{equation}
	\end{enumerate}
\end{thm}

\subsection{Random pseudometrics.}
 Let $$T: (\RR^d)^2 \to \RR_+$$ be a random pseudometric, that is 
$T$ satisfies the axioms of a metric except the non-degeneracy. 
Recall that a $\QQ$-\emph{semi-norm} over $\RR^d$ is a map $\nu \to \R^d \to \R_+$ satisfying $$\forall (\lambda,x)\in \QQ\times \R^d, \mu(\lambda x) = |\lambda|\mu(x),$$
and $\forall (x,y)\in (\R^d)^2, \ \nu(x+y)\leq \mu(x)+\mu(y).$
\begin{thm}\label{fixeddirT}
	Let $T$ be a random  pseudometric satisfying ~(\ref{homT}) (ergodicity)  and condition~(\ref{moment}) (finite moment). 
	Then, there exists a $\QQ$-semi-norm $\mu$  $\mu: \RR^d \to \R_+$ such that
	\begin{equation}\label{convergencesT}
	\forall v\in \RR^d, \lim\limits_{n \to +\infty} \frac{1}n T(0,nv) = \mu(v)\ \text{  almost surely and $L^1$}.
	\end{equation}
	If  $T$ satisfies the further condition~(\ref{isoT})  (isotropy) then $\mu$ is constant over $\mathbb S^{d-1}$.
\end{thm}
Note that a semi-norm over $\R^d$ is always continuous. 
As in the discrete case, the proof relies only on the ergodicity of the field, see \S~\ref{exicon}.	
The main result of this paper  is the following:
\begin{thm}\label{positifT}
	Let  $T $ be a random pseudometric over $\RR^d$ satisfying conditions~(\ref{homT}) (ergodicity),  (\ref{moment}) (finite moment), (\ref{mesuT}) (annular mesurability),  ~(\ref{annuliT}) (decay of instant one-arms) and~(\ref{AI}) (quasi independence). Then
	$ \mu$ is a norm.
\end{thm}
Before going on, we would like to make some remarks. 
\begin{rem}
	\begin{itemize}
		\item  We emphasize that this theorem is general, and does not deal with the particularities of the model. This is the reason we can apply it to such different models as Gaussian fields,  Voronoi percolation, Boolean percolation or smooth random metrics.
		\item Theorem~\ref{positifT} relies on the two crucial conditions~(\ref{annuliT}) (decay of instant one-arms) and~(\ref{AI}) (asymptotic independence). The first condition is obtained for free in the case of random smooth metrics, see~\S \ref{smooth}. For our three percolation settings, these conditions are easy to prove, or rely on recent known results.
		\item The second condition needs exponentially small asymptotic dependence, which is the reason why for Gaussian percolation we need fields with exponentially fast decorrelation, and why our results for Riemannian metrics need either finite correlation, or in the planar Gaussian case, exponentially small dependence. This is also the reason why our result recovers only partially the Boolean case, see \S~\ref{boole}. Notice that this condition enables us to deal with infinite correlations and to have an alternative to the Van den Berg-Kesten (BK) inequality, which is a crucial tool for percolation in independent settings. 
	\end{itemize}
\end{rem}
In~\cite{gouere2017positivity}, J. B. Gou\'er\'e and M. Th\'eret proved the following:
\begin{thm}\label{muzeroT}\cite[\S 2]	{gouere2017positivity}
	Let  $T$ be a random pseudometric over $\R^d$ satisfying conditions~(\ref{homT}) (ergodicity), (\ref{mesuT}) (annular mesurability),  (\ref{lipT}) (Lipschitz), ~(\ref{isoT}) (isotropy)  and~(\ref{vannuliT}) (instant crossings of large annuli). Then $\mu=0$.
\end{thm}
The aforementioned article~\cite{gouere2017positivity} is written for Boolean percolation, but the proof holds in our context. We explain it in \S~\ref{subz}.

\paragraph{The ball shape theorem.}  
Theorem~\ref{positifT} is  extended into the ball shape theorem, the exact counterpart of Theorem~\ref{ballshape}. For this, 
for any $t\geq 0$ let us define
\begin{eqnarray}\label{KK}
B_t&=& \{x\in \RR^d, T(0,x)\leq t\}\\
 \text{ and }
K&=&\left\lbrace x\in \RR^d, \mu(x)\leq 1\right\rbrace,
\end{eqnarray}
where $\mu$ is defined by (\ref{convergencesT}).

\begin{thm}\label{ballshapeBFT}
	Let $T$ be a random pseudometric over $\RR^d$ satisfying~(\ref{homT}) (ergodicity) and (\ref{lipT}) (Lipschitz). 
	\begin{enumerate}
		\item If $\mu=0$
		then for any positive $M$, 
		$$ \PP\Bigl[
				M\mathbb B \subset \frac{1}t B_t
				\text{ for all $t$ large enough}
				\Bigr] =1.$$
		\item If $\mu$ is a norm then 
		$K$ is a convex  compact subset of $\RR^d$ with non-empty interior. Besides, for any positive $\epsilon$,
		\begin{equation}\label{inclusions}
		\PP\Bigl[(1-\epsilon)K\subset\frac{1}tB_t\subset(1+\epsilon)K\text{ for all $t$ large enough}\Bigr]=1.
		\end{equation}
		If  $T$ further satisfies condition~(\ref{isoT}) (isotropy), then $K=\frac{1}{\mu(1)}\mathbb B$, where $\mathbb B\subset\RR^d$ denotes the unit ball and $\mu(1)$ denotes $\mu(v)$ for any vector $v$ of norm 1.
	\end{enumerate}
\end{thm}

\paragraph{Critical exponents.}
As a direct consequence of Theorem~\ref{positifT}, 
we can prove that if a model has vanishing time constants, 
then the probability that there are long instant paths cannot decay too fast:
\begin{cor}\label{cegT}
		Let  $T $ be a random pseudometric over $\RR^d$ satisfying conditions~(\ref{homT}) (ergodicity),  (\ref{moment}) (finite moment), (\ref{mesuT}) (annular mesurability)  and~(\ref{AI}) (quasi independence). Assume also that $\mu=0$, where $\mu$ is the pseudo-norm defined by Corollary~\ref{fixeddirT}. Then 
	$$ \forall \eta>0, \ \limsup_{R\to \infty} R^{d+\eta}\PP[ T(A_R)=0] >0,$$
where $T(A_R)$ denotes the pseudo-distance between the two sphere $S(0,1)$ and $S(0,R)$, see~(\ref{spsh}) and~(\ref{TAR}). 
\end{cor}
When $T$ is defined through the integral of a random non-negative function, it implies that the probability that there are null long paths cannot decay too fast, see Corollary~\ref{ceg} and in particular Corollary~\ref{ceb} in the standard Bernoulli case over $\ZZ^d$.

\subsection{Random densities.}

\paragraph{General setting.} We now introduce a very general and natural family of examples over $\RR^d$, namely pseudometrics generated by \emph{random densities}, that is random non-negative functions of $\RR^d$.
Let 
$$\sigma: \RR^p \to \R_+$$ be a random measurable function over $\RR^d$ with non-negative values.
We define an analogue of the discrete almost metric~(\ref{defFPPmetric}).
For any $x,y$ in $\RR^d$:
\begin{equation}\label{defmetric}
T(x,y):=\inf_{\substack{\gamma\text{ piecewise } \Cc^1 \\ \text{ path } x \rightarrow y}}\int_\gamma \sigma.
\end{equation}
Then, $T(x,y)$ is the least time to travel from $x$ to $y$. As a consequence, $T$ is a pseudometric, possibly with infinite values; as in the discrete setting, $T$ is not a distance in general, since $T$ can vanish at a pair of different points. 
Indeed over a domain where $\sigma=0$, then $T$ vanishes.  The points where $\sigma=0$ are called \emph{white} points.

As a particular but very natural case, a \emph{random colouring} $\sigma$ has values in $\{0,1\}$. In this case, 
we travel over $\{\sigma=1\}$ with speed one and  with infinite speed over $\{\sigma=0\}$.  
\begin{rem}\label{remL}
Note that if $\sigma$ is bounded, in particular if $\sigma$ is a colouring, then its associated pseudometric $T$ satisfies automatically the condition~(\ref{lipT}) ($T$ Lipschitz).
\end{rem}

\paragraph{The time constant for densities.}
We now provide results for the associated FPP. These results need conditions, which 
are satisfied in our four applications, namely Bernoulli percolation, Gaussian fields, Voronoi percolation and Boolean percolation, the latter with a further condition. 
The existence of the time constant is a consequence of Theorem~\ref{fixeddirT}:
\begin{cor}\label{fixeddirS}
Let $\sigma :\RR^d \to \R_+$ be a random  density satisfying conditions~(\ref{mesu}) (mesurability), (\ref{momentS}) (finite moment) and~(\ref{hom}) (ergodicity). 
\begin{enumerate}
	\item 
Then, there exists a pseudo-norm $\mu$ satisfying the convergence property~(\ref{convergencesT}).
\item If $\sigma$ satisfies the further condition~(\ref{iso}) (isotropy), then $\mu$ is constant over $\mathbb S^{d-1}$.
\end{enumerate}
\end{cor}
\begin{rem}\label{remM} If $\sigma$ is bounded, for instance if $\sigma$ is a colouring, then condition(\ref{momentS}) (finite moment) is always satisfied. 
	\end{rem}

Theorem~\ref{positifT} implies the following important result:
\begin{cor}\label{positif}
Let  $\sigma: \RR^d\to \R_+$ be a random density satisfying conditions~(\ref{mesu}) (mesurability),  (\ref{momentS}) (finite moment), (\ref{hom}) (ergodicity), ~(\ref{annuliT}) (decay of instant one-arms) and~(\ref{AI}) (quasi independence). Then
$ \mu $ is a norm.
\end{cor}
Theorem~\ref{muzeroT} implies the following:
\begin{cor}\label{muzero}
	Let  $\sigma: \RR^d \to \R_+$ be a random density satisfying
	conditions~(\ref{mesu}) (mesurability), (\ref{momentS}) (finite moment), (\ref{hom}) (ergodicity), ~(\ref{iso}) (isotropy), (\ref{lipT}) ($T$ Lipschitz) and~(\ref{vannuli}) (white crossings of large annuli). Then $\mu=0$.
\end{cor}
This corollary needs the isotropy of $\sigma$, which is too much asked for the planar colouring examples we have in mind, see Corollary~\ref{remtassion}. We provide another criterion.  
\begin{prop} \label{muzero2}
		Let  $\sigma: \RR^d \to \{0,1\}$ be a random colouring satisfying
	conditions~(\ref{mesu}) (mesurability), ~(\ref{hom}) (ergodicity) and~(\ref{wRSW}) (weak Russo-Seymour-Welsh). Then $\mu=0$.
\end{prop}
 
Theorem~\ref{ballshapeBFT} implies the following 
\begin{cor}\label{ballshapeBF}
		Let  $\sigma: \RR^d \to \R_+$ be a random density satisfying
	conditions~(\ref{mesu}) (mesurability), (\ref{momentS}) (finite moment), (\ref{hom}) (ergodicity), and~(\ref{lipT}) ($T$ Lipschitz).
	\begin{enumerate}
			\item If $\mu=0$, then for any $M>0$, 	$$ \PP\Bigl[
			M\mathbb B \subset \frac{1}t B_t
			\text{ for all $t$ large enough}
			\Bigr] =1.$$
		\item If $\mu$ is a norm, then 
$K$ is a convex  compact subset of $\RR^d$ with non-empty interior. Besides, for any positive $\epsilon$,
		$$\PP\Bigl[(1-\epsilon)K\subset\frac{1}tB_t\subset(1+\epsilon)K\text{ for all $t$ large enough}\Bigr]=1.
		$$
		If  $\sigma$ satisfies the further condition~(\ref{iso}) (isotropy), then $K=\frac{1}{\mu(1)}\mathbb B$.
	\end{enumerate}
\end{cor}

Proposition~\ref{muzero2} and the first assertion of Corollary~\ref{ballshapeBF} have a nice corollary for \emph{planar} colourings,  using~\cite{tassion2016crossing}. 
			\begin{cor}\label{remtassion} 
				Let $\sigma: \R^2 \to \{0,1\}$ be a planar random colouring satisfying conditions~(\ref{mesu}) (mesurability), ~(\ref{hom}) (ergodicity), (\ref{FKG}) (FKG), (\ref{col}) (colour invariance) (\ref{wsym}) (weak symmetries) and~(\ref{sq}) (crossing of squares).
			  Then $\mu=0$ and the balls defined by $T$ grow faster than the Euclidean ones. 
					\end{cor}
				
				\begin{rem}\label{Xor}
					\begin{enumerate}
						\item   By Remarks~(\ref{remL}) and~(\ref{remM}), if $\sigma$ is bounded, then it satisfies  conditions~(\ref{momentS}) (finite moment) and (\ref{lipT}) ($T$ Lipschitz). This holds in particular for random colourings.
						\item 
					If $\sigma$ is also  bounded below by a positive constant, and fullfills conditions~(\ref{mesu}) and~(\ref{hom}), then it satisfies the hypotheses of  Corollary~\ref{fixeddirS}, Corollary~\ref{positif} and the second case of Corollary~\ref{ballshapeBF}.
\end{enumerate}
				\end{rem}

\paragraph{Critical exponents.}
Corollary~\ref{cegT} has the following implication for random densities. 
\begin{cor}\label{ceg}
	Let $\sigma: \RR^d \to \R_+$ be a random density satisfying conditions~(\ref{mesu}) (mesurability), (\ref{momentS}) (finite moment), (\ref{hom}) (ergodicity), (\ref{AI}) (quasi independence). Assume also that $\mu=0$, where $\mu$ is the pseudo-norm defined by Corollary~\ref{fixeddirS}. Then 
	$$ \forall \eta>0, \ \limsup_{R\to \infty} R^{d-1+\eta}\PP[\Cross_0 (A_R)] >0,$$
	where $\Cross_0(A_R)$ denotes the event that there is a white path between the two spheres $S(0,1)$ and $S(0,R)$, see~(\ref{croc}).
	\end{cor}
This corollary means in particular that if a strongly decorrelating percolation model has a polynomial decay for the one-arm, then the critical exponent is less than $d-1$.

\subsection{Assumptions }

\paragraph{Notations. }
\begin{itemize}
	\item 
	The set $\mathcal T$  of pseudometrics (resp. the set $\mathcal F$ of real functions) over $\R^d$ is equipped with the natural partial order $\leq$.
	An event $E$ in $\mathcal C$ (resp. $\mathcal F$) is said to be  \emph{increasing} if 
	$$\varphi \in E\text{ and } \varphi\leq\psi \Rightarrow \psi\in E.$$
	An event is \emph{decreasing} if 
	\begin{equation}\label{dek}
	\varphi\in E\text{ and } \varphi\geq\psi \Rightarrow \psi\in E.
	\end{equation}
	\item For any pair of subsets $A, B\subset \R^d$, 
	let $\mathcal A^-$ and $\mathcal B^-$  the set of decreasing events in $\mathcal T$ (resp. $\mathcal F$) depending only on the values of $T\in \mathcal T$ (resp. $f\in \mathcal F$) over $A$ and $B$ respectively.
For any positive $ Q,S$ let
\begin{equation}\label{ind}
\Ind^- (Q,S):= \sup_{A, B \subset \R^d, \Diam A \leq S, \Diam B \leq S \atop \dist(A,B)>Q, E_A\in \mathcal A^-, E_B\in \mathcal B^-} \left|\PP[ E_A\cap E_B]-\PP[E_A]\PP [E_B]\right|.
\end{equation}
\end{itemize}

\begin{itemize}
	\item For any $0<r< R$, denote by
	$A_{r,R}$ and $A_R$ the spherical shells 
	\begin{eqnarray}\label{spsh}
	A_{r,R}&= &B(0,R) \setminus B(0,r)\subset \RR^d \\ \nonumber
	A_R &= &A_{1,R}. \\
	\label{TAR}
	T(A_{r,R})&=& \inf_{x\in S(0,r),  y\in S(0,R)} T(x,y).
	\end{eqnarray}
	\item 		For every $v\in \RR^d$, $\tau_v$ denotes the translation associated with $v$. The translations of $\RR^d$ act on the set $\mathcal T(\RR^d)$ of pseudometrics over $\RR^d$ by
	\begin{equation}\label{actionT}
	\forall v\in \RR^d, \forall T\in \mathcal T (\R^d), \forall (x,y)\in \RR^d, \ \tau_v(T)(x,y) = T (v+x,v+y).
	\end{equation}
The action $\tau_v$ is said to be \emph{ergodic}
	for the law of the pseudometric $T$ is invariant under the action $\tau_v$, and if 
	for any event $A$, if $A$ is invariant under $\tau_v$ then $A$ has measure $0$ or $1$.
\end{itemize}

\paragraph{Conditions for random pseudometrics.}
In the sequel, $T$ denotes a pseudometric on $\R^d$.
\begin{itemize}
		\item Assumptions used for the existence of $\mu$ (Theorem~\ref{fixeddirT})
\begin{enumerate}[resume=condi]
	\item \label{homT} (Ergodicity) $T$ is ergodic under the action of the translations of $\RR^d$.
	\item \label{moment} (Finite moment) For any $x\in \RR^d$, 
	$\mathbb E \left(T(0,x)\right)$ is finite. 
	\end{enumerate}
	\item Assumptions for the positivity of $\mu$ (Theorem~\ref{positifT})
	\begin{enumerate}[resume=condi]
			\item \label{mesuT} (Annular mesurability) For any $0<r<R$, $T(A_{r,R})$ is mesurable with respect to the $\Sigma-$algebra of the random pseudometrics $T$.
	\item \label{annuliT} (Decay of instant one-arms) There exists $\eta, R_0>0$, such that 
	$$\forall R\geq R_0, \ \PP \left[ T(A_{R})=0\right] \leq \frac{1}{R^{d-1+\eta}}.$$
	\item \label{AI}(Exponential quasi independence) 
	There exists a positive constant $\beta$ such that for any $\alpha>0$, there exists $ Q_0$ such that for any $Q\geq Q_0$,
	$$
	\Ind^-(Q,Q^{1+\alpha}) \leq  \exp(-Q^\beta),
	$$
	where $\Ind^-$ is defined by~(\ref{ind}).
	\end{enumerate}
\item Assumptions for the vanishing of $\mu$ (Theorem~\ref{muzeroT})
\begin{enumerate}[resume=condi]
	\item \label{vannuliT} (Instant crossings of large rescaled annuli)
	$$ \limsup_{R\to \infty} \PP[T(A(R,2R))=0] >0.$$
		\item\label{isoT} (Isotropy) The measure of $T$ is invariant under the action of the orthogonal group of $\R^d$.
	\item (Lipschitz) \label{lipT} There exists a positive $C>0$ such that $T$ is $C$-Lipschitz for the Euclidean metric. 
	\end{enumerate}
\end{itemize}
\paragraph{Comments for the general conditions.}
\begin{itemize}
	\item  Conditions~(\ref{homT}) (ergodicity) and (\ref{moment}) (finite moment) are needed by Kingman's ergodic Theorem~\ref{subadditive} and the existence of the time constant, see Theorem~\ref{fixeddirT}. 
	\item Condition~(\ref{annuliT}) (decay of instant one-arms) is one of the two crucial assumptions needed for our main Theorem~\ref{positifT} (positivity of $\mu$). This fact is intuitive: if the travelling time for crossing an annulus is too small, then it is believable that the time constant will drop to zero. 
	\item Notice also that Condition~(\ref{annuliT}) has an optimal flavour. Indeed, for Bernoulli percolation and a lot of other percolation models like planar Gaussian or Voronoi, the decay of white one-arms at criticality is polynomial. For planar Bernoulli percolation over the triangular lattice, the exponent is $5/48$~\cite{lawler2002}, to be compared to our bound $1$ for $d=1$.
	\item Condition~(\ref{AI}) (asymptotic independence) is the other crucial assumption needed for the main Theorem~\ref{positifT}: it ensures weak dependency between two disjoint parts of the space. In classical FPP, the so-called BK inequality gives the good comparison of disjoint events. Here, because we handle colouring with possibly infinite correlations, condition~(\ref{AI}) is a way to replace BK.
	\item In fact, our condition~(\ref{AI}) could be weakened in considering only events which are finite intersections of events of the type $\{T(A_R)<\delta\}$, see the proof of Proposition~\ref{comparison}.
	\item 
	Condition~(\ref{vannuliT}) (instant crossings of large annuli) is one of the conditions needed by Theorem~\ref{muzeroT}, which asserts that the time constant vanishes. It is  trivially satisfied by random Riemannian metrics. 
	\item Condition~(\ref{lipT}) (Lipschitz) is needed for the ball shape theorem and for Theorem~\ref{muzeroT}. It is satisfied for all our applications, except for Riemannian FPP. For the latter case, \cite{lagatta2010shape} proves however a ball shape theorem. 
	  	\item The last condition~(\ref{isoT}) (isotropy) is needed for Theorem~\ref{muzeroT}. It is not clear that it is really necessary. It is also needed to prove that, in the positive case, the limit ball $K$ given by Theorem~\ref{ballshapeBFT} is a Euclidean ball. 
\end{itemize}
\paragraph{Conditions for random densities and colourings.}
We now specify conditions for the density setting. 
For this we will need further notations and definitions:
\begin{itemize}
\item The translations over $\RR^d$ act on the set $\mathcal D(\R^d)$ of densities of $\R^d$ by 
\begin{equation}\label{action}
\forall v\in \RR^d, \forall \sigma \in \mathcal D(\R^d), \tau_v(\sigma) = \sigma \circ \tau_v,
\end{equation}
where $\tau_v$ denotes the translation associated with $v$. The action $\tau_v$ is said to be \emph{ergodic}
for the law of the random density $\sigma$ if the latter is invariant under the action $\tau_v$, and if 
for any event $A$, if $A$ is invariant under $\tau_v$ then $A$ has measure $0$ or $1$.
\item For any $0<r< R$, let
\begin{equation}\label{croc}
\Cross_0(A_{r,R})=\left\{\exists \text{ a $C^0$ path included in $\{\sigma=0\}$ crossing $A_{r,R}$ }\right\}.
\end{equation}
	\item Assume that $\sigma:\R^d\to \{0,1\}$ is a random colouring. For any right parallelipiped  $R=\prod_{i=1}^N [a_i,b_i] \subset \RR^d$, where $a_i<b_i$ for every $i$, define for any $j\in \{0,1\}$,
	\begin{eqnarray}\label{crossd}
	\Cross_j(R)&:=&\big\{\text{There exists a $C^0$ path in $\{\sigma =j\}\cap R$ intersecting}\nonumber
	\\ &&\{a_1\}\times \prod_{i=1}^N [a_i,b_i]\text{ and } \{b_1\}\times \prod_{i=1}^N [a_i,b_i]\big\}.
	\end{eqnarray}
	For $d=2$, this is just the classical lenghtwise crossing of a rectangle.
	We will call by slight abuse this type of crossing a \textit{black} crossing when $j=1$ and a \textit{white} crossing when $j=0$.
\end{itemize}
We can now state conditions for a random density $\sigma$ that ensure that the associated pseudo-distance $T$ satisfies the general conditions described in the previous paragraph.
\begin{itemize}
	\item Assumptions used for the existence of $\mu$ (Corollary~\ref{fixeddirS})
	\begin{enumerate}[resume=condi]
		\item \label{mesu}(mesurability) Almost surely, $\sigma$ is mesurable.
		\item \label{momentS} (finite moment) For any $x\in \R^d$, 
		$$\mathbb E \int_{[0,x]} \sigma <+\infty.$$
		\item \label{hom}(Ergodicity) 
		The	 translations of $\RR^d$ are ergodic for the  law of $\sigma$. 
	\end{enumerate}
	\item Assumptions for the positivity of $\mu$ (Corollary~\ref{positif})
	\begin{enumerate}[resume=condi] 
		\item \label{wannuli} (Decay of white one-arm) There exist $\eta, R_0>0$, such that for any $R\geq R_0$, $$ \PP \left[ \Cross_0(A_R)\right] \leq \frac{1}{R^{d-1+\eta}}.$$ 
	\end{enumerate}
	\item Assumptions used for the vanishing of $\mu$ (Corollary~\ref{muzero} and Proposition~\ref{muzero2})
	\begin{enumerate}[resume=condi]
						\item \label{iso} (Isotropy) The measure of $\sigma $ is invariant under the orthogonal group of $\R^d$. 
		\item \label{vannuli} (White crossings of large annuli) There exists $c>0$ such that
		$$ \limsup_{R\to \infty}\PP \left[ \Cross_0(A_{R,2R})\right] \geq c.$$
			\item\label{RSW} (Russo-Seymour-Welsh) Assume here that $\sigma$ is a colouring. 
		\begin{enumerate}
			\item \label{wRSW} (weak RSW)
			For any $d-$uple of closed non trivial intervals $I_1, \cdots, I_d$, there exist $c>0$, such that 
			$$ \limsup_{n\to \infty}  \PP \left[ \Cross_0(n \prod_i I_i)\right] \geq c.$$
			\item \label{sRSW} (strong RSW) 	For any $d-$uple of closed non trivial intervals $I_1, \cdots, I_d$, there exist $c>0$, such that 
			$$ \liminf_{n\to \infty}  \PP \left[ \Cross_0(n \prod_i I_i)\right] \geq c.$$
			
		\end{enumerate}
	\end{enumerate}
	\item Secondary assumption
	\begin{enumerate}[resume=condi]
		\item\label{blackreg} (Positive region regularity) Almost surely, $\{\sigma>0\}\subset \RR^d$ is a locally finite union of $d$-dimensional submanifolds with piecewise $C^1$ boundary, such that for any pair $(W_1,W_2)$ of these submanifolds, $\overline{ W_1} \cap \overline{W_2} = \emptyset$, or $W_1\cap W_2$ contains an open subset.
	\end{enumerate}
\item Assumptions for Corollary~\ref{remtassion}. Assume here that $\sigma$ is a colouring.
\begin{enumerate}[resume=condi]
	\item \label{col} (Colour invariance) The law of $\sigma$ is invariant under change of colour.
	\item \label{wsym}(Weak symmetries) The law of $\sigma$ is invariant under 
	right-angle rotation, under symmetries by horizontal axis.
	\item \label{sq} (Crossing of squares) There exists $c>0$ such that for any square 
	$S$, $$\PP[\Cross_0(S)] >c.$$
		\item \label{FKG} (Fortuin-Kasteleyn-Ginibre inequality for crossings) For any positive crossing events $E_1$ and $E_2$ of the form $\Cross_1(R)$, 
	$$ \PP[E_1 \cap E_2]\geq \PP[E_1]\PP[E_2].$$	
\end{enumerate}
\end{itemize}
\paragraph{Comments for the colouring conditions.}
\begin{itemize}
	\item  Condition~(\ref{mesu}) (mesurability) is necessary for the definition~(\ref{defmetric}) of the associated pseudometric 
	$T$. Note that a priori $T$ can be infinite. 
	Condition~(\ref{hom}) (ergodicity) implies condition~(\ref{homT}) (ergodicity) for $T$.
	\item 
	Condition~(\ref{annuliT}) (decay of instant one-arms) implies condition~(\ref{wannuli}) (decay of white one-arms), but Lemma~\ref{3a3b}  implies that the converse if true, if the geometric condition~(\ref{blackreg}) (positive region regularity) is satisfied. These conditions are satisfied by our $C^1$ Gaussian fields,  Voronoi percolation and Boolean percolation, our main applications, see Corollary~\ref{equiarm} and Proposition~\ref{equiVor}. The main asset of condition~(\ref{wannuli}) is that it concerns the percolation properties of $\sigma$, and not its FPP properties. 
		\item Condition~(\ref{iso}) (isotropy) implies that $T$ satisfies condition~(\ref{isoT}) (isotropy), needed for Corollary~\ref{muzero}.
		\item 
Condition~(\ref{vannuli}) (white crossings of large annuli) implies condition~(\ref{vannuliT}) (instant crossings of large annuli), whereas condition~(\ref{wRSW}) (weak RSW) implies condition~(\ref{vannuli}). The latter condition is defined in all dimensions,  however the only examples we know are two-dimensional, and its higher dimension version will not be used in this paper. 
	\item Condition~(\ref{FKG}) (FKG) is needed only in the two-dimensional situation of Corollary~\ref{remtassion}, which proves the vanishing of the time constant in the general situation  of~\cite{tassion2016crossing}. 
\end{itemize}

\section{Applications}\label{applications}

We present the various applications of  Theorem~\ref{positifT} to Bernoulli, Gaussian, Voronoi and Boolean percolations, and then to 
Riemannian FPP.

\subsection{Classical FPP}
We can reprove the hardest half of Theorem~\ref{tp}.
\begin{cor}\label{corclassical} If $\PP[\nu=0]<p_c(d)$, then $\mu_\nu$ is a norm.
\end{cor}
However, we obtain a new result:
\begin{cor}\label{cebp}
	Let $\sigma_{p_c}: \mathbb E^d\to \{0,1\}$ the critical bond Bernoulli percolation. Then 
	$$\forall \eta>0  \  \limsup_{R\to \infty} R^{d-1+\eta}\PP[\Cross_0 (A_R)] >0.$$
\end{cor}
In this Bernoulli critical percolation, it is known that
\begin{enumerate}
	\item \cite{kozma} for $d>19$ for bond percolation over $\mathbb L^d$ (and others lattices with enough symmetries)
	$ \PP[\Cross_0 (A_R)] \asymp \frac{1}{R^2}$;
	\item  \cite{lawler2002} for $d=2$, 
	$ \PP[\Cross_0 (A_R)] = \frac{1}{R^{5/48(1+o(1))}}$ for the site percolation over the triangular lattice. 
	\item \cite[(5.1)]{kesten1987scaling} For bond percolation over $\ZZ^2$, $\PP[\Cross_0 (A_R)]\geq \frac{C}{R^{1/3}}$.
\end{enumerate}
After a first version of this article, Vincent Beffara and Hugo Vanneuville told us that Corollary~\ref{cebp} can be proved more directly. We give the argument of H. Vanneuville since it holds for correlated fields, see \S~\ref{proofclass}.

\subsection{Gaussian FPP}
Continous Gaussian fields are very natural object in probability. Gaussian percolation, which can be defined by the connectivity features of the associated nodal domains, that is the subset of points where the function is positive, has recently  become a very active domain.
\paragraph{Setting and former results.}
Let $$f: \RR^p \to \RR$$ be any Gaussian field. To this field we associate a family $(\sigma_p)_{p\in \RR}$ of colouring functions over $\RR^p$ defined by: 
\begin{equation}\label{sip}
\forall p\in \RR, \ \sigma_p:= \frac{1}2\left(1+\text{sign} (f+p)\right),
\end{equation} where the sign is considered as $-1$ over $\{f=0\}$. This choice will have no influence if $f$ satisfies condition~(\ref{sr}) (strong regularity), see Theorem~\ref{regular}. 

The first main application of Corollary~\ref{positif}, that is Theorem~\ref{positifT} for colourings, can be viewed as the natural sequel of two recent theorems which exhibit strong similarities beween two models of very different nature, namely  the sign of a smooth isotropic planar Gaussian field on one side, and Bernoulli percolation on the other.
Firstly, in~\cite{beffara2016v}, V. Beffara and the second author of this work proved a Russo-Seymour-Welsh theorem for the nodal domains $\{f>0\}$: 
\begin{thm}\label{RSW_2}\cite{beffara2016v}
	Let $f$ be a centered  smooth Gaussian field on $\RR^2$ with non-negative smooth covariance kernel $e$
	depending only on the distance, with polynomial decay with degree at least 325 and such that $e=1$ on the diagonal. Let $\sigma_0$ be the	associated colour function defined by~(\ref{sip}) for $p=0$.Then, 
	\begin{enumerate}
		\item
	for any rectangle $R\subset \RR^2$,  
	$$\liminf_{n\to \infty} \PP[\Cross_0(nR)]>0.$$
	\item There exists $C,\alpha>0$, such that
	$$\forall R\geq 1, \PP[ \Cross_0(A_R)]\leq \frac{C}{R^\alpha}.$$
\end{enumerate}
\end{thm}
With the definitions given above and below, Theorem~\ref{RSW_2} says that $\sigma_0$ satisfies the strong Russo-Seymour-Welsh condition~(\ref{sRSW}).
The second assertion implies that there is no infinite component of $\{f>0\}$, a negative result which 
was already in~\cite{alexander1996boundedness}, with a different (sketched) proof. 
Secondly, in~\cite{rivera2017critical}, A. Rivera and H. Vanneuville proved that for the Bargmann-Fock field~(\ref{BF}) below the value $p=0$ is critical:
\begin{thm}\cite{rivera2017critical} \label{uniqcomp}
	Let $f:\R^2 \to \R$ be the Bargmann Fock field~(\ref{BF}).
	\begin{enumerate}
		\item 
	If $p\geq 0$, then a.s. there is no unbounded connected component of $\{\sigma_p=0\}$. 
	\item If $p<0$, then a.s. there is a unique unbounded connected component of $\{\sigma_p=0\}$.
	\end{enumerate}
\end{thm}
\begin{rem}
	\begin{enumerate}
		\item Theorem~\ref{RSW_2} was followed by several improvements on the decay condition, see~\cite{beliaev2018discretisation} (degree 16), \cite{rivera2019quasi} (degree 4) and \cite{muirhead2018sharp} (degree 2).
		\item Theorem~\ref{uniqcomp} was also impoved, see 
		\cite{muirhead2018sharp} (polynomial decay with degree 2 and condition~(\ref{sp}) and \cite{rivera2019talagrand} (degree 2 with condition~(\ref{wp})), see also~\cite{garban2019bargmann} with another proof.
		\item These two results are Gaussian equivalents to classical percolation results on lattices, see~\cite{grimmett}. 
		\item The positivity assumption on the kernel is essential in these results, since by Theorem~\ref{pitt}  it satisfies the FKG condition~(\ref{FKG}), thanks to which we can use a general theorem by V. Tassion~\cite{tassion2016crossing}.
			\end{enumerate} 
\end{rem}

\paragraph{Gaussian FPP.}
The first main consequence of the general Corollary~\ref{positif} concerns planar Gaussian fields:
\begin{thm}\label{posBF}
Let $f$ be a centered Gaussian field over $\RR^2$ and satisfying assumptions (\ref{symmetries}) (stationarity), (\ref{sr}) (strong regularity) and (\ref{wd}) (weak decay of correlations).  
Let $(\sigma_p)_{p\in \RR}$ the associated family of colour functions given by~(\ref{sip}). Then, 
\begin{enumerate}
	\item the associated family of time functions $(\mu_p)_{p\in \RR} $ defined by Corollary~\ref{fixeddirS}, associated with the family of pseudometrics defined by~(\ref{defmetric}), are well defined;
	 	\item the conclusions of Corollary~\ref{ballshapeBF} (ball shape theorem) hold.
	\item  Assume that $f$ satisfies the further condition~(\ref{wp}) (weak positivity of correlations). Then, $$p\leq  0 \Rightarrow \mu_p=0.$$
\item Assume that $f$ satisfies the further conditions~(\ref{wp}) (weak positivity of correlations) and~(\ref{ed}) (strong decay of correlations). Then,
$$
 \mu_p >0  \Leftrightarrow  p>0.$$
\end{enumerate}
\end{thm}
\begin{rem}\label{higher}
	The two first assertions (existence of $\mu$ and the ball shape theorem) hold in higher dimensions with the same conditions. 
	\end{rem} 
	All these assumptions are satisfied by a particular Gaussian field called the \emph{Bargmann-Fock field}, which also satisfies condition~(\ref{isotropy}) (isotropy). This field arises naturally from random complex and real algebraic geometry as explained in \cite{beffara2016v}. It is given by the non negative correlation function:
$$e(x,y)=\exp\left(-\frac{1}{2}\|x-y\|^2\right).
$$
Equivalently, we can explicitly write it as the following random field $f$:
\begin{equation}\label{BF}
f(x)=\exp\left(-\frac{1}{2}\|x\|^2\right)\sum\limits_{i,j\in\NN} a_{i,j}\frac{x_1^ix_2^j}{\sqrt{i!j!}},
\end{equation}
where the $a_{i,j}$'s are i.i.d centered Gaussians of variance $1$. 

\paragraph{One-arm exponent.}

Corollary~\ref{ceg} has the following corollary:
	\begin{cor}\label{cegaussian}
Let $f$ be a centered Gaussian field over $\RR^2$ and satisfying assumptions (\ref{symmetries}) (stationarity), (\ref{sr}) (strong regularity), (\ref{wp}) (weak positivity of correlations) and~(\ref{ed}) (strong decay of correlations). For $p=0$, that is the colouring function is $\sigma_0$, then 
		$$ \forall \eta>0, \ \limsup_{R\to \infty} R^{1+\eta}\PP[\Cross_0 (A_R)] >0.$$
	\end{cor}
In particular, the degree $\alpha$ in Theorem~\ref{RSW_2} satisfies $\alpha\leq 1$.

\subsection{Voronoi FPP. }

\paragraph{Setting and former results.}
The second application concerns Voronoi percolation.  
Let $X$ be a Poisson process over $\RR^d$ with intensity $1$. Recall that $X$ is a random subset of points, locally finite, such that for any Borel subset $A\subset \RR^d$, 
the probability that $X\cap A$ has exactly $k $ points equals
$$\frac{(\Vol A)^k}{k!} \exp(-\Vol A).$$ Moreover, for two disjoint subsets $A$ and $B$, $X_{|A}$ is independent of $X_{|B}$.
To $X$ 
 we can associate the so-called \emph{Voronoi 
tiling}: any point $x$ of $X$ has a cell $V_x\subset \RR^d$ defined by the points in $\RR^d$ which are closer to $x$ 
than any other point of $X$. Then, we colour any cell in black (value $1$)
with probability $1-p$ or in white (value $0$) with probability $p$. 
The boundaries of two cells with different colour are coloured white. 
This provides a random colouring
$$ \sigma_p: \RR^d \to \{0,1\}.$$
Let $p_c(d)\in [0,1] $ be defined
by 
\begin{equation}\label{pcdvor}
p_c(d)=\sup \left\{p, \text{there exists an infinite white component a.s.}\right\}.
\end{equation}
It is classical~\cite[pp. 270--272]{bollobas2006percolation} that for any $d\geq 2$, 
$p_c(d)\in]0,1[.$
In 2006, B. Bollob\`as and O. Riodan proved:
\begin{thm}\cite[Theorems 1.1 and 1.2]{bollobas2006critical}\label{BoRi} For Voronoi 
	percolation, 
$p_c(2) =1/2.$
\end{thm}
Then V. Tassion proved that at criticity, planar Voronoi percolation $\sigma_0$ satisfies a Russo-Seymour-Welsh type theorem:
\begin{thm}\cite[Theorem 3]{tassion2016crossing}
	\label{RSWbevo}
	If $p=p_c(2)=1/2$, the planar Voronoi percolation satisfies condition~(\ref{sRSW}) (strong RSW).
	Morover, there exists $C,\alpha>0$ such that 
	$$ \forall R\geq 1, \ \PP[\Cross_0( A_R)]\leq \frac{C}{R^\alpha}.$$
	\end{thm}
Note that in \cite{bollobas2006critical} the weaker condition~(\ref{wRSW}) (weak RSW) was proved.
More recently, H. Duminil-Copin, A. Raoufi and V. Tassion proved the following result:
\begin{thm}\cite[Theorem 1]{duminil2019}\label{DRT}
For any $p\in [0,1]$, let $\sigma_p$ be the Voronoi 
percolation model defined above. For $p<p_c$, there exists $c>0$ and $R_0>0$, such that 
	$$\forall R\geq R_0,\  \PP[\Cross_0(A_R)]\leq \exp(-c R).$$ 
	In particular, $\sigma_p$ satisfies condition~(\ref{wannuli}) (decay of white one-arm).
\end{thm}
For $d=2$, it was already proved by~\cite[Theorem 1.2]{bollobas2006critical}.

\paragraph{Voronoi FPP. }
We will see that these results together with our general Corollary~\ref{positif} imply our second main application:
\begin{thm}\label{bevo}
	For any integer $d\geq 2$ and $p\in [0,1]$, let $\sigma_p$ be  the Voronoi 
	percolation model defined above. Then,
	\begin{enumerate}
\item 	 the associated time constant $\mu_p$ defined by (\ref{convergences})  is well defined. 
\item The following holds:
	\begin{eqnarray*}
	p<p_c(d)&\Rightarrow&	 \mu_p >0 \\
	\text{and } \mu_p >0 &\Rightarrow&	p\leq p_c(d).
	\end{eqnarray*}
	\item 
	For $d=2$, $$ \mu_p >0 \Leftrightarrow p< \frac{1}2.$$
\item 	Corollary~\ref{ballshapeBF} (ball shape theorem) applies, and the convex $K$ is a an Euclidean ball.
\end{enumerate}
\end{thm}
\begin{rem}
	There exist other models of FPP for Voronoi tesselations, see~\cite{howard1997euclidean} and~\cite{pimentel2006time}. The first one always gives positive times, and the second one is associated with the graph given by the tesselation. 
	\end{rem}

\paragraph{One-arm exponent.}

Corollary~\ref{ceg} has the following corollary:
\begin{cor}\label{ceV}
Let $\sigma_{1/2}: \R^2 \to \{0,1\}$ be  the planar critical Voronoi 
percolation model defined above. Then,
	$$ \forall \eta>0, \ \limsup_{R\to \infty} R^{1+\eta}\PP[\Cross_0 (A_R)] >0.$$
\end{cor}

\subsection{Boolean FPP.}\label{boole}

\paragraph{Setting and former results.}
One classical continuous FPP model is the so-called \emph{Boolean} or \emph{continous percolation}, where Euclidean balls of random radii centered at points of a random Poisson process of intensity $\lambda$ on $\RR^d$ are painted in white, and the rest of the space in black (this is the inverse of the classical colours; this colouring fits our general model above). 
It provides a random colouring
\begin{equation}\label{sil}
\sigma_{\nu, \lambda}: \RR^d \to \{0,1\},
\end{equation}
where $\nu$ is the radius law. It is known~\cite[Proposition 7.3]{meester1996}  that $\nu $ satisfies the condition~(\ref{ww}) below if and only if $(\sigma_{\nu, \lambda})_{\lambda}$ is a non-trivial family of Boolean percolations, which means  that in the case this condition is not fullfilled, for any $\lambda>0$, almost surely there the union of balls covers $\R^d$.
 
In 2017, J.-B. Gou\'er\'e and M. Th\'eret proved  that Theorem~\ref{fixedFPP} (existence of the time constant) holds for Boolean percolation, and more importantly, that Theorem~\ref{tp} (phase transition for the Bernoulli FPP) has an analogue in the Boolean setting. For this, define for a given radius law $\nu$:
\begin{equation}\label{boocrit}
\hat \lambda_c(\nu, d):=\sup\left\{
\lambda\geq 0,  \PP \left[\Cross_0(A_{R,2R})\right] \to_{R\to \infty} 0\right\},
\end{equation}
where $\Cross_0(A_{R,2R})$ denotes the probability that there exists a white continuous path from $S(0,R)$ to $S(0,2R)$, see~(\ref{croc}).
Besides, consider three conditions for the radius law $\nu$.
\begin{enumerate}[resume=condi]
	\item \label{ww} (optimal moment condition)  $\mathbb E_\nu (r^d)<\infty$.
	\item \label{winu}(weak moment condition) $$\int_0^\infty \left(\PP_\nu([r,+\infty[)\right)^{1/d} dr < \infty.$$
	\item \label{exptail} (exponential small tail) There exists $c>0$, such that 
	$$\forall r\geq 1, \ \PP_\nu( [r, +\infty[) \leq \exp(-cr). $$
	\end{enumerate}
Another slightly more natural threshold is defined by
\begin{equation}\label{boocrit2}
\lambda_c(\nu, d):=\sup\{\lambda\geq 0, \text{ there is no infinite white component a.s.}\}.
\end{equation} It is easy to see that $\lambda_c \leq \hat \lambda_c$. 
In~\cite[Theorem 2.1]{gouere2008} J.-B. Gou\'er\'e proved that $\lambda_c>0$ if and only if condition~(\ref{ww}) is satisfied. Under a little stronger condition for $\nu$, it was proved in~\cite{duminil2018subcritical} 
that $\lambda_c= \hat \lambda_c.$ In dimension 2, \cite{ahlberg2018}, this was previously obtained with the optimal condition~(\ref{ww}). 
In~\cite{gouere2017positivity} the authors proved the following:
\begin{thm}\label{gouthe}\cite[Theorem 1.2]{gouere2017positivity} Let $\lambda>0$, $\nu$ be a radius law satisfying condition~(\ref{winu}) and $\sigma_{\nu,\lambda}: \RR^d \to \{0,1\}$ the associate Boolean percolation colouring. Then,  
	$$ \mu_{\nu, \lambda} \text{ is a norm } \Leftrightarrow \lambda< \hat\lambda_c,$$
	where $\mu_{\nu, \lambda}$ is the semi-norm associated to $\sigma_{\nu, \lambda}$ defined by Corollary~\ref{fixeddirS}.
\end{thm}
Corollary~\ref{positif}
can reprove Theorem~\ref{gouthe} in the restrictive case where the law for the radii satisfies condition~(\ref{exptail}), which ensures condition~(\ref{AI}) (quasi independence). 
\begin{cor}\label{bouboule} For any density $\lambda>0$ and radius mesure $\nu$ satisfying condition~(\ref{exptail})
	Then, $$  \mu_{\nu,\lambda} \text{ is a norm } \Leftrightarrow \lambda < \lambda_c.$$
	\end{cor} 
 As for Voronoi percolation FPP, it uses a recent result by H. Duminil-Copin, A. Raoufi and V. Tassion:
 \begin{thm}\cite[Theorem 2]{duminil2018subcritical}\label{boc}
 	Assume $\nu$ satisfies condition~(\ref{exptail}).
Then for any $\lambda<\lambda_c$, there exists $c>0$ and $R_0>0$, such that 
 	$$\forall R\geq R_0, \  \PP[\Cross_0(A_R)]\leq \exp(-cR).$$
 	In particular, $\sigma_{\nu, \lambda}$ satisfies condition~(\ref{wannuli}) (decay of white one-arms). 
 	\end{thm}
 Clearly, our method, even with the further condition~(\ref{exptail}), does not reach the simplicity
  of~\cite{gouere2017positivity}, which does not need Theorem~\ref{boc}. However it illustrates again the generality of the passage
  	from percolation to FPP, as long as we have the good decay of the one-arm and strong decorrelation. However, the following corollary seems new. 
 
 \paragraph{Critical exponent}
 
 	\begin{cor}\label{ceb}
 	Fix a radius mesure $\nu$ satisfying condition~(\ref{exptail}) (exponentially small tail), and let $\sigma_{\nu, \lambda_c}$ be the associated random critical Boolean colouring defined by~(\ref{sil}), where $ \lambda_c(\nu, d)>0$ denotes the critical density defined by~(\ref{boocrit2}) associated with $\nu$.  
 	Then, 
 	$$\forall \eta>0,\  \limsup_{R\to \infty} R^{d-1+\eta}\PP[\Cross_0 (A_R)] >0.$$
 \end{cor}

\subsection{Riemannian FPP}	\label{smooth}

\paragraph{Setting and former result. }
We follow the setting of~\cite{lagatta2010shape}. 
Denote by $\text{Sym}^{+}(d)$ the open cone of 
positive real symmetric matrices of size $d$. Then 
any continuous map
$$g: \RR^d \to \text{Sym}^{+}(d)$$
equips $\RR^d$ with a continuous Riemannian metric, 
as well as a distance $T: \RR^d \to \R_+$ defined by 
\begin{equation}\label{cmet}
\forall (x,y)\in (\R^d)^2, \ 
T(x,y) = \inf_{\gamma \text{ piecewise } C^1\atop
	x\to y} \Leng_g (\gamma),
\end{equation}
where, if $\gamma: [a,b] \to \RR^d$ is a $C^1$ path, 
$$ \Leng_g (\gamma) = \int_a^b \|\gamma'(s)\|_{g(\gamma(s))} ds.$$
We will need the following various assumptions:
\begin{enumerate}[resume=condi]
	\item (Ergodicity) \label{gerg} The random Riemannian metric $g$ is ergodic,that is its law is invariant under the orthogonal group of $\R^d$, and the invariant events are of probability one or 0.
	\item (Regularity)\label{greg} Almost surely, $g$ is continuous.
	\item (Finite range)\label{FR} There exists $Q>0$, such that the values of $g$ at any pair of points at distance at least $Q$ are independent.
	\item (Finite moment)
	\begin{enumerate}
		\item (strong)
	\label{sFM}  for any $v\in \RR^d$, for any $r\in \RR$, 
	$\mathbb E (e^{r\Lambda_0})$ is finite, where $\Lambda_0$ is the largest eigenvalue of $g(0).$  
	\item (weak)\label{wFM}
	for any $v\in \RR^d$, 
	$\mathbb E (\|v\|_{g(0)}) $ is finite.
	\end{enumerate}
\item (Increasing length) \label{IL} If the set of random metrics $g$ is equipped with a partial order, then $g\mapsto \Leng_g $ is increasing. 
\end{enumerate}

In~\cite{lagatta2010shape}, the authors proved the following:
\begin{thm}\cite[Theorem 2.5]{lagatta2010shape}\label{smoothmetricLW} Let $g: \RR^d \to \text{Sym}^+(d)$ be a random Riemannian metric  satisfying conditions~(\ref{gerg}) (ergodicity), (\ref{greg}) (regularity), ~(\ref{sFM}) (strong finite moment condition) and~(\ref{FR}) (finite range).
	Then, Theorem~\ref{fixeddirT} applies, and the pseudo-norm $\mu$ associated with the distance $T$ is a norm. 
\end{thm}
Note that even if $T$ is a distance, it could happen that $\mu$ degenerates.
\paragraph{New results. }
Theorem~\ref{positifT} also applies in this context and gives the following:
\begin{thm}\label{smoothmetric0} Let $g: \RR^d \to \text{Sym}^+(d)$ be a random Riemannian metric satisfying conditions~(\ref{gerg}) (ergodicity), (\ref{greg}) (regularity), ~(\ref{wFM}) (weak finite moment condition). Assume moreover that $T$ satisfies condition~(\ref{AI}) (asymptotic independence). 
	Then, the pseudo-norm $\mu$ associated with the distance $T$ is a norm. 
\end{thm}

As a corollary, we reprove Theorem~\ref{smoothmetricLW} with a milder condition:
\begin{cor}\label{smoothmetric}
	Theorem~\ref{smoothmetricLW} holds, replacing condition~(\ref{sFM})(strong finite moment condition) by (\ref{wFM}) (weak finite moment condition).
\end{cor}
When $d=2$ and for metrics $g$ induced by Gaussian fields, we can deal with infinite correlations under a further hypothesis:
\begin{cor}\label{smoothmetric2} Let $f$ be a centered Gaussian field over $\RR^2$ satisfying assumptions (\ref{symmetries}) (stationarity), (\ref{sr}) (strong regularity) and (\ref{ed}) (strong decay of correlations).  Let $g: \RR^d \to \text{Sym}_d^+$ be a random metric induced by $f$, such that $g$ satisfies conditions~(\ref{gerg}) (ergodicity), ~(\ref{greg}) (regularity),  (\ref{wFM}) (weak finite moment condition)  and~(\ref{IL}) (increasing length) for the partial order induced by the one for $f$. 
	Then, the pseudo-norm $\mu$ associated with the distance $T$ is a norm. 
\end{cor}
Note that in this case, we do not need positive correlations for the Gaussian field. As a family of examples, we apply this corollary to planar conformal random metrics induced by strongly decorrelationg Gaussian fields. For this, we need to 
define a new pair of conditions for real deterministic functions:
\begin{enumerate}[resume=condi]
	\item \label{phi0} (Increasing) $\varphi: \R\to \R_+^*$ is a continuous positive non-decreasing map; 
	\item \label{phi1} (Weak integrability)
	$\int_0^\infty \varphi (u) e^{-\frac{u^2}2}du <\infty.$
\end{enumerate}
\begin{cor}\label{smoothmetric3}
	Let $f$ be a centered Gaussian field over $\RR^2$ and satisfying assumptions (\ref{symmetries}) (stationarity), (\ref{sr}) (strong regularity), and (\ref{ed}) (strong decay of correlations).  Let $g: \RR^d \to \text{Sym}_d^+$ be a random metric defined by
	$$ g(f):= \varphi(f) g_0,$$
	where $g_0$ is the standard metric over $\RR^2$ and $\varphi: \R\to \R^*_+$ satisfying conditions~(\ref{phi0}) and (\ref{phi1}).
	Then, the pseudo-norm $\mu$ associated with the distance $T$ is a norm. 
	\end{cor}
\begin{rem}
	\begin{enumerate}
\item	In	\cite[Theorem 3.1]{lagatta2010shape}, the authors also proved a ball shape theorem. The proof of our general ball shape Theorem~\ref{ballshapeBFT} needs  $T$ to be Lipschitz, which is not the case here. However, it should be possible to weaken this condition.
\item Another model of smooth metrics has been provided in~\cite{ferrari2013random}. These are K\"ahler metrics defined over compact complex manifolds. A natural question is to prove a version of Theorem~\ref{smoothmetric0} in this context, at least over the projective space for polynomials of increasing degree in the spirit of \cite{beliaev2017russo}, which uses the symmetries of the sphere instead of the symmetries of the plane as in~\cite{beffara2016v}, or over $\mathbb C^n$ for the semi-classical rescaled limit. Note that these metrics are given by the second derivatives of random holomorphic function and have infinite correlations, a double difficulty. However, since the model is based on the Bargmann-Fock model, there is some hope. 
\end{enumerate}
\end{rem}

\subsection{Other models}

\paragraph{Another Gaussian pseudometric}
For Gaussian fields, it is very natural to generalize the pseudometric associated with the colouring $\sigma_p$. Indeed, let 
$$\psi: \R\to \R_+$$ any map such that
\begin{enumerate}[resume=condi]
	\item \label{psi0} (Increasing) $\psi$ is  non-decreasing ; 
	\item \label{psi1} (Flat negative sea) $\forall x\in \R, \psi(x)>0 \Leftrightarrow x>0;$
	\item \label{psi2} (Weak integrability)
	for any positive $\alpha,$  $\int_0^\infty \psi (u) e^{-\alpha u^2}du <\infty.$
\end{enumerate}
For any random function $f:\RR^d \to \R$, define the random density:
\begin{equation}\label{gauss2metric}
\sigma =\psi\circ f,
\end{equation}
and $T$ the associated pseudometric defined by~(\ref{defmetric}).
\begin{thm}\label{gauss2} Let $f: \RR^d \to \R$ be a Gaussian field satisfying the hypotheses of Theorem~\ref{posBF}, $\psi: \R \to \R_+$ be satisfying the conditions~(\ref{psi0}), (\ref{psi1}) and (\ref{psi2}). 
	For any $p\in \R$, denote by 
	$\sigma_{p}$ the random density defined by~(\ref{gauss2metric}) associated with $f+p$, that is $\sigma_p = \psi\circ (f+p)$. Then, the conclusions of Theorem~\ref{posBF} hold for the 
	associated time constant $\mu_{p}$.
\end{thm}
Note that Theorem~\ref{posBF} is a particular case of Theorem~\ref{gauss2}, choosing $\psi = {\bf 1}_{\R_+}$. Another interesting natural choice is given by $\psi = {\bf 1}_{\R_+}\text{Id}  $. 

\paragraph{Ising model.}
The Ising model does not belong to the core of this paper, since we do not have positive time constant in this situation. Consequently, we refer for instance to~\cite{Velenik} for definitions and classical properties. Corollary~\ref{remtassion} and~\cite{beffara2017percolation} have the following consequence:
\begin{cor}\label{ising} There exists $\beta_0<0$ such that the following holds. Let $s_\beta$ be the Ising model over the triangular lattice, with temperature $T=1/\beta$, and denote by $\beta_c>0$ the critical parameter. Let $\sigma_\beta$ be the associated random planar colouring, where the dual hexagons are painted with the value of the center spin. Assume that $\beta\in ]\beta_0, \beta_c[$. Then $\mu_\beta=0$, 
	where $\mu_\beta$ is the time constant defined by~(\ref{convergences}), and associated with $\sigma_\beta$.
\end{cor}

\subsection{Assumptions for Gaussian percolation.}
\paragraph{Setting.  }
We consider a centered Gaussian field $f$ on $\RR^2$, i.e a random field on $\RR^2$ such that for any finite set of points $(x_1,\cdots,x_n)\in\RR^2$, the vector $\big(f(x_1),\cdots, f(x_n)\big)$ is a centered Gaussian vector.
Although of course all the definitions hold in higher dimensions, our results concern essentially only two-dimensionsal fields (see however Remark~\ref{higher}). For this reason we restrict ourselves to $d=2$. 
Recall that $f$ is entirely determined by its {covariance kernel}:
$$\forall (x,y)\in (\RR^2)^2, \ e(x,y):=\mathbb E \left(f(x),f(y)\right).$$
We will assume that the covariance is stationary, that is invariant under translations, so that there exists $\kappa: \RR^2 \to \RR$ such that 
\begin{equation}\label{kappa}
\forall x,y\in \RR^2, \ e(x,y)=\kappa(x-y).
\end{equation}
We will also assume that $\kappa$ is a continuous function, hence by Bochner's theorem one can define its
spectral measure $dm$ by 
$ \kappa = \mathcal F[m]$.
 We will assume that $dm$ is uniformly continuous with respect to the Lebesgue measure, so that there exists $\rho: \RR^2\to \RR$
 such that $dm(x)=\rho^2(x) dx. $ 
  Since $\kappa(0) = \int \rho^2 dx$, $\rho$ is $L^2$, thus it has a well-defined Fourier transform. 
 As in~\cite{muirhead2018sharp} and \cite{rivera2019talagrand}, let $q: \RR^2 \to \RR$ be defined by 
\begin{equation}\label{q}
	q:= \mathcal F[\rho],
\end{equation}
where $\mathcal F$ denotes the Fourier transform. In this case, 
$$
\kappa= q\star q,
$$
and $f$ can be expressed as the convolution of $q$ and the white noise, but we won't use this in this paper. 

\paragraph{Assumptions.}

We will need the field to satisfy part or  all of the following assumptions, as in the two aforementioned papers. 
\begin{enumerate}[resume=condi]
	\item\label{symmetries}  (Symmetries) The Gaussian field $f$	 is centered, its covariance $e$  is stationary, 
 and normalized, 
with $\kappa(0)=1$, where $\kappa$ is defined by~(\ref{kappa}) above. 	Moreover, the function $\kappa$ is symmetric under both  reflection in the
$x$-axis, and  rotation by $\pi/2$ about the origin.
	\item\label{deux} (Regularity). 
	\begin{enumerate}
		\item \label{wr} (weak) $\kappa$ is continuous. 
\item \label{sr} (strong)	The function $q$ of Definition~\ref{q} is in $L^2(\RR)$, $C^3$, even, and the support of $\rho$ contains an neighbourhood of $0$.   
	\end{enumerate}
	\item \label{quatre} (Positive correlations)
\begin{enumerate}
	\item\label{wp}	(weak) $\kappa  = q \star q \geq 0 $ 
	\item \label{sp} (strong) $q\geq 0.$
	\end{enumerate}
	\item \label{cinq} (Decay of correlations). 
	\begin{enumerate}
		\item \label{wd}(weak) $\kappa(x)\to_{\|x\|\to \infty} 0$.
		\item \label{ed} (strong)
	There
	exists two positive constants $C, \beta$ such that for every multi-index $\alpha$ with $|\alpha| \leq 3,$
	$$\forall x\in \RR^2, \ |\partial^\alpha  q(x)| \leq  C\exp(-\|x\|^\beta) .$$
	\end{enumerate}
	\item \label{isotropy} (Isotropy). The kernel depends only on the distance between two points. 
\end{enumerate}
\paragraph{Comments on the Gaussian conditions. }
\begin{itemize}
\item We begin with simple assumptions.  As explained at the beginning of this paragraph, the stationarity in condition~(\ref{symmetries}) and condition~(\ref{wr}) allows to  define the various objects, $\kappa$, $\rho$ and $q$. Stationarity and condition~(\ref{wd}) imply ergodicity, so that the colour function $\sigma_p$ defined by~(\ref{sip}) will satisfy condition~(\ref{hom}) (ergodicity). 
Condition~(\ref{wr}) implies that $\sigma_p$ satisfies condition~(\ref{mesu}) (mesurability). 
The other symmetries are needed for Theorems~\ref{expdec} (decay of white one-arms)  and~\ref{muirhead2018sharp} (strong decorrelation of events) below. The latter theorems also need condition~(\ref{sr}).
\item 
The positivity condition~(\ref{wp}) is important and needed in Theorem~\ref{expdec} (decay of white one-arms). It  implies FKG inequality, see Theorem~\ref{pitt}, which is crucial in this kind of work.
Note that Theorem~\ref{expdec} ensures that 
$\sigma_p$ satisfies conditions~(\ref{wannuli}) (decay of white one-arms), see Corollary~\ref{expdecann}.
The stronger condition~(\ref{sp}) is here for a historical remark given below. 
\item Condition~(\ref{ed}) (exponential decay of correlations) is also  important and used in Theorem~\ref{muirhead2018sharp} which ensures an exponential decay of the correlation between monotonic events. This theorem implies that $\sigma_p$ satisfies condition~(\ref{AI}) (quasi independence), see Corollary~\ref{coroind}.  Recall that this condition is needed to counterbalance some combinatorial term growing exponentially fast with the observed scale, see (\ref{strass}) in Proposition~\ref{comparison}.
\item Condition~(\ref{ed}) is also used in Theorem~\ref{expdec}. However, in its original form, this theorem only needs a polynomial decay with small degree.       
\item Condition~(\ref{isotropy}) is needed only to prove that in the case of positive time constant, the pseudo-ball $K$ of Corollary~\ref{ballshapeBF} defined by the pseudometric $T$ related to $\sigma_p$ is an Euclidean ball. 
\item Finally let us point out that the full list of conditions from~(\ref{symmetries}) to~(\ref{isotropy}) are satisfied by the  Bargmann-Fock field~(\ref{BF}).
\end{itemize}

\section{Proof of the general theorems}\label{secproof}

\subsection{Existence of the time constant}\label{exicon}

In this subsection, we prove Theorem~\ref{fixeddirT} (existence of the time constant). The main tool is Kingman's subadditive ergodic Theorem. 
Before its statement, we recall some elements of ergodicity. 
Let $(X,\mathbb P)$ a probability space, and $F: X\to X$ be a mapping preserving the measure. Then $F$ is said to be \emph{ergodic} if 
for any event $A$ which is invariant under $F$, $A$ has measure $0$ or $1$.
\begin{thm}\label{subadditive}\cite[Theorem 3.3.3]{viana_oliveira_2016}
Let $(X,\PP)$ be a probability space. 
Let $F: X \to X$ be a measurable transformation of $X$, ergodic for $\PP$.
Let $(\varphi_n)_{n \in \NN}$ be a family of real non-negative valued random variables such that
$\EE\varphi_1$ is finite and 
\begin{equation}\label{subadd}
\forall n,m\in \NN, \forall x\in X, \ 
\varphi_{n+m}(x) \leq \varphi_n(x)+\varphi_m(F^n(x)).
\end{equation}
Then there exists $\mu$ in $[0,\infty)$ such that:
$$
\lim_{n \to +\infty} \frac{1}n\varphi_n = \mu \text{  almost surely and $L^1$}.
$$
\end{thm}
\begin{proof}[Proof of Theorem~\ref{fixeddirT}]
Let $\mathcal T$ be the set of realisations of the pseudometrics $T$. 
Let $v\in \mathbb S^{n-1}\subset \RR^d$, $\tau_v$ the translation by $v$, which acts on $\mathcal T $ as~(\ref{actionT}). Then, by condition~(\ref{homT}) (ergodicity), $\tau_v$ is ergodic for the law of $T$.
For any $n\in \NN$, let 
$$\varphi_n:=T(0,nv).$$
By the triangle inequality for $T$, the subadditivity inequality~(\ref{subadd}) holds. 
Nonnegativity of $T$ is trivial, and the finiteness of the expectation $\mathbb E(\varphi_1)$ is provided by condition~(\ref{moment}).
Hence by Theorem~\ref{subadditive} there exists $\mu(v)\in \RR_+$, such that 
$$\frac{1}n T(0,nv) \to_{n\to \infty} \mu(v)$$
almost surely  an $L^1$.
Moreover for $t\in \QQ$, we have:
\begin{equation}\label{grou}
\lim\limits_{n\to+\infty} \frac{1}n T(0,nt v) =|t|\lim\limits_{x\to\infty} \frac{1}n T(0,nv)=|t|\mu(v).
\end{equation}
Now,  
$$ \forall x,y\in \RR^d,  \ \forall n\in \NN, \ T(0,ny)\leq T(0,nx)+T(nx,ny).$$
Dividing by $n$ and taking the average, by the invariance under translations
of the law of $T$, this gives:
\begin{equation}\label{pog}
\mu(y)-\mu(x) \leq \mu(x-y),
\end{equation}
which implies, using that $\mu$ is even,
$ \mu(x+y)\leq \mu(x)+\mu(y).$ We thus proved that $\mu$
is a semi-norm. 
\end{proof}
\begin{lem}\label{mulip}
	Under the hypotheses of Theorem~\ref{fixeddirT}, 
	the semi-norm $\mu$ is Lipschitz.
	\end{lem}
\begin{proof} 
	Since $\mu$ is a semi-norm, 
	 $$ \forall x\in \RR^d, \ \mu(x) \leq \sum_{i=1}^d |x_i| \mu(e_i)\leq \|x\|_1 \max_i \mu(e_i),$$
	where $(e_i)_i$ denotes the standard basis of $\RR^d$
	and $x=\sum_i x_i e_i$.
	\end{proof}

\subsection{Positivity of the time constant}\label{posit}
 Theorem~\ref{positifT}, which asserts that $\mu$ is a norm if the ergodic pseudometric $T$ satisfies condition~(\ref{annuliT}) (decay of instant one-arms) and~(\ref{AI}) (quasi-independence), 
is a consequence of the following Proposition~\ref{keyprop}.
Recall that for any $M>1$, $A_M$ denotes the spherical shell centered at $0$ of inner radius $1$ and outer radius $M$, see~(\ref{spsh}), and $T(A_M)$ denotes the minimal time of a path from the interior sphere to the outside of the shell $A_M$, see~(\ref{TAR}).
\begin{prop}\label{keyprop}
Let $T: (\RR^d)^2 \to \R_+$ be a random pseudometric satisfying conditions~(\ref{homT}) (ergodicity), (\ref{mesuT}) (annular mesurability),  ~(\ref{annuliT}) (decay of instant one-arms) for $\eta>0$ and~(\ref{AI}) (quasi-independence).
Then, there exists an unbounded positive increasing sequence $(M_n)_n$ and a positive number $c$ such that
$$
 \forall n\in \NN, \ \PP\left[\frac{T(A_{M_n})}{M_n}<c\right]\leq \frac{1}{M_n^{d-1+\eta}}.
$$
\end{prop}
\begin{rem}\label{large}  Note that a large deviation result would suffice to get an exponential decay. 
	 It is possible that the wide applicability range of Theorem~\ref{positifT} is a consequence of the leniency of this result. Moreover, its provides Corollary~\ref{cegT}, which was the best general result known for percolation. \rk{à voir} 
	\end{rem}
Given this proposition, we can prove the main Theorem~\ref{positifT}.
\begin{proof}[Proof of Theorem~\ref{positifT}.]
Let $v\in \mathbb S^{d-1}$ and $(M_n)_n$ the sequence given by~Theorem~\ref{keyprop}. By Theorem~\ref{fixeddirT} there exists a constant $\mu(v)\geq 0$ such that 
$$
 \frac{1	}{\lfloor M_n\rfloor +1}T(0,(\lfloor M_n\rfloor +1)v)\xrightarrow[n\to\infty]{L^1}\mu(v).
$$
Since for any $n$, 
$
T(A_{M_n})\leq T(0,(\lfloor M_n\rfloor +1)v),$ the latter limit and Proposition~\ref{keyprop}
imply that $\mu(v)\geq c$ and thus $\mu(v)>0$. 
\end{proof}
Proposition~\ref{keyprop} will be proved by induction over scales. However we will need to renormalize the constant $c$, see Corollary~\ref{corocomp}
below. 
To this end, we begin by proving the following 
Proposition~\ref{comparison} which compares the crossing time probabilities of two spherical shells with different exterior radii. 
\begin{prop}\label{comparison}
Let $T$ be a random pseudometric over $\RR^d$ satisfying assumption~(\ref{homT}) (ergodicity) and (\ref{mesuT}) (annular mesurability).
Then, 
for any $1\leq Q<R<S$ and any positive constant $\delta$,
\begin{equation}\label{strass} 
\PP\left[\frac{T(A_S)}S< \frac{\delta}{1+\frac{Q}{R}}\right]\leq  
\left(c_d S^{d-1}\frac{R}{Q}\right)^n
\left(  \PP\left[\frac{T(A_R)}R< \delta \right]^{n}
+n \Ind^-( Q,S)
\right),
\end{equation}
where $c_d>0$ is a constant depending only on the dimension $d$,  where 
 $ n=\lfloor N\frac{Q}{2R+2Q}\rfloor $  with $N= \lfloor \frac{S-1}{2R+Q}\rfloor $, and where $\Ind^-$ is defined by~(\ref{ind}).
\end{prop}
\begin{proof}
Let $\{B_1,\cdots,B_{N}\}$ be a maximal set of disjoint spherical shells, centered on $0$, included in $A_S$, of increasing radii, of width $2R$, such that the interior sphere of $B_1$ is the unit sphere, and separated by a sequence $(C_1,\cdots,C_{N})$ of spherical shells centered on $0$, of width $Q$ and of increasing radii, see Figure~\ref{path}.
	 We have 
	 $$N= \left\lfloor \frac{S-1}{2R+Q}\right\rfloor.$$
	  For any $j\in \{1, \cdots , N\}$,  we consider a minimal set of $k_j$  translates of $A_{R}$ inside $B_j$, such that the closure of the union of their interior  disks contains the middle sphere of $B_j$, that is $S\left(0,1+(j-1)(2R+Q)+R\right)$. These conditions ensure that any continuous crossing of $B_j$ crosses at least twice at least  one of the $k_j$ copies of $A_R$ inside $B_j$. It is true that there exists $c_d>0$ depending only on the dimension $d$, such that
	  \begin{equation}\label{sumkj}
	  \forall 1\leq j\leq N, \ k_j\leq c_d S^{d-1}.
	  \end{equation}
	  	  Let $\gamma$ be  a minimizing path across the shell $A_S$. 
	  	   By the previous remark, $\gamma$ necessarily crosses one copy of $A_{R}$ in each $B_j$, once to enter the interior ball, and then once more to leave it. It thus crosses at least $N$ such shells, each of them twice. We call $(a_1,\cdots,a_{N})$ the sequence of the first copies of $A_{R}$ it crosses in each $B_j$, see Figure~\ref{path}.
	  	   We have 
	  	   \begin{equation}\label{sss}
	  	   T(A_S)\geq \sum_{j=1}^{N}2 T(a_j).
	  	   \end{equation}
	  	    Now, for any $j\in \{1, \cdots, N\}$, consider the corrresponding event:
	\begin{figure}
		\centering
		\includegraphics[width=8cm]{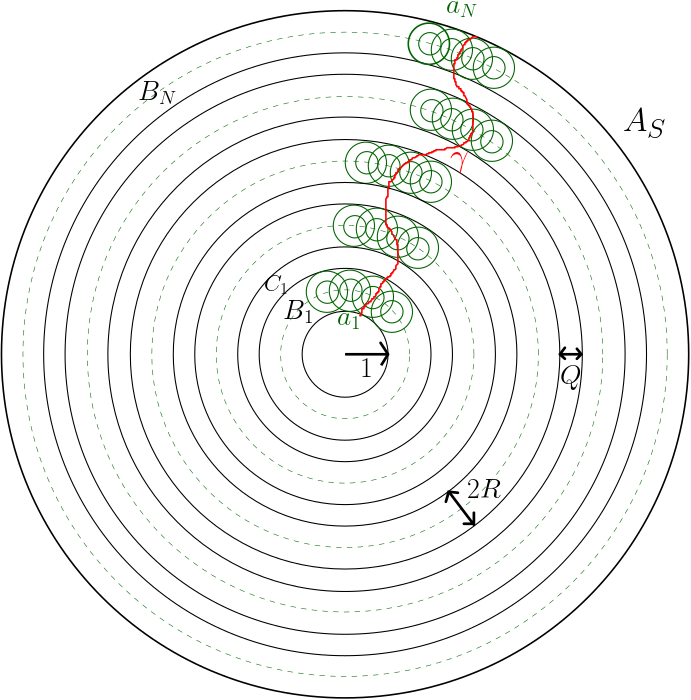}
		\caption{The path $\gamma$ going across $A_S$ crosses a certain number of copies of $A_R$.}
		\label{path}
	\end{figure}
	$$
	E_j:=\left\{\frac{T(a_j)}{R}<\delta (1+\frac{Q}R)\right\}.
	$$
When  the event $\left\{\frac{T(A_S)}S<\delta \right\}$ occurs, at least 
$$n=\lfloor N\frac{Q}{2R+2Q}\rfloor $$ events of the form $E_j$ occur. 
	Indeed, otherwise we would have by~(\ref{sss})
	$$
	\frac{T(A_S)}S \geq 2\frac{R}S (N-N\frac{Q}{2Q+2R}) \delta (1+\frac{Q}R)\geq \delta.
	$$
	Assume from now on that $n\geq 1$. Note that if $n=0$ then (\ref{strass}) is trivially true. 
	Using~(\ref{sumkj}),
	$$
	\PP\left(\frac{T(A_S)}{S}<\delta \right)\leq 
	{N \choose n }(c_d S^{d-1})^n  
	\sup_{\substack{a_1,\cdots,a_ {n}\text{ copies of }A_{R}
			\\
			\text{ on disjoint shells $B_j$}}}\PP\left[\bigcap\limits_{j=1}^{n}E_j\right].
	$$ 
	Indeed, there is ${N \choose n}$ ways to choose the $n$ annuli $(B_{j_1}, \cdots B_{j_n})$ where $E_{j_1}, \cdots, E_{j_n}$ happen, and for any $i=1, \cdots, n$, there is at most $c_dS^{d-1}$ choices for the small 
	annulus $a_{j_i}$. Now, given such a deterministic sequence $a_1,\cdots,a_{n}$, since by definition the distance between any two of the shells $B_j$ is at least $Q$, the distance between any two of the $a_j$'s has the same lower bound. 	
By definition of $\Ind^-$, using the fact that a finite intersection
of $E_j$'s is a decreasing event, 
for all $S>R> Q\geq 1$, 
$$\forall i\in \{1, \cdots, n\}, \ 
\PP\left[E_i\cap \bigcap_{j=i+1}^{n} E_j\right]\leq \PP[E_j]\PP\left[\bigcap_{j=i+1}^{n} E_j\right]
+ \Ind^-(Q,S).
$$
By an immediate induction, this implies
$$
	\PP\left[\bigcap_{j=1}^{n}E_j\right]\leq\left(\PP\left[E_1\right]\right)^{n}
	+ n \Ind^-(Q,S).
	$$
	By 
	the classical inequality 
	\begin{equation*}
\forall 1\leq n\leq N, 	{N \choose n}\leq
	\left(\frac{eN}n\right)^n
	\end{equation*}
and the definition of $n$,	the combinatorial term satisfies
	$$ {N \choose n }(c_d S^{d-1})^n \leq \left(\frac{4c_d S^{d-1}R}{Q}\right)^n.$$
		Replacing $\delta $ with $\delta(1+Q/R)^{-1}$, we obtain the result.
		\end{proof}
In the next Corollary~\ref{corocomp}, Proposition~\ref{comparison} 
is applied to a sequence of growing scales, threatening the inductive renormalized
constant $\delta$ to drop to zero. However, the sequence is chosen so that
the infinite product of the renormalization factors converges to a positive constant.
\begin{cor}\label{corocomp}
	Let $\eta>0$ and $T$ be a random pseudometric satisfying assumptions~(\ref{homT}) (ergodicity), (\ref{mesuT}) (annular mesurability) and (\ref{AI}) (quasi-independence) for $\beta>0$.
Let
		$$\epsilon = \frac{1}2\min (1, \frac{\eta}d, \beta).$$  
		Then there exists $R_0>0$, such that 
	for any positive constant $\delta$ and any $R\geq R_0$,
	\begin{equation}\label{eq}
 \PP\left[\frac{T(A_R)}{R}< \delta \right]\leq \frac{1}{R^{d-1+\eta}} \Rightarrow 
	\PP\left[\frac{T(A_{R^{1+\epsilon}})}{R^{1+\epsilon}}< \frac{\delta}{1+R^{-\frac{\epsilon}2}}\right]\leq  
	\frac{1}{(R^{1+\epsilon})^{d-1+\eta}}.
	\end{equation}	
\end{cor}
\begin{proof}
	We apply Proposition~\ref{comparison} with
	$(Q,R,S)= (R^{1-\frac{\epsilon}2}, R, R^{1+\epsilon})$
	so that there exists $R_1>0$ depending only on $\eta$ and $\beta$, such that for $R\geq R_1$,
	$$ \frac{R^\epsilon}{4}\leq  N\leq \frac{R^\epsilon}{2}, \ 
	\frac{R^{\frac{\epsilon}2}}{8}\leq  n\leq \frac{R^{\frac{\epsilon}2}}{4}
	\text{ and }
	\left(c_dS^{d-1}\frac{R}{Q}\right)^n \leq  
	(c_dR^{(d-1)(1+\epsilon)+\frac{\epsilon}2})^{n}.$$
	Hence, by Proposition~\ref{comparison} and condition~(\ref{AI}),  there exists $R_2\geq R_1$ depending
	only on $\eta$ and $\beta$, such that 
	for any $R\geq R_2$, if the left-hand side of~(\ref{eq}) holds, then
	$$	\PP\left[\frac{T(A_{R^{1+\epsilon}})}{R^{1+\epsilon}}
	< \frac{\delta}{1+R^{-\frac{\epsilon}2}}\right]
	\leq 
	(c_dR^{(d-1)(1+\epsilon)+\frac{\epsilon}2})^{n}
	\big( R^{-(d-1+\eta)n}+ R^{\frac{\epsilon}2}e^{-R^{(1-\frac{\epsilon}2)\beta}}\big).
	$$
Hence, 
	 there exists $R_3\geq R_2$ depending only on $d$, $\beta$ and  $\eta$, 
	such that for any $R\geq R_3$, 
	the right-hand side is bounded above by 
	$ R^{-(d-1+\eta)(1+\epsilon)}.$ 
\end{proof}

To implement the implication~(\ref{eq}), we  need to find a scale where the left-hand side holds. This is done by the following lemma:
\begin{lem}\label{cross} Let $T $ be a random pseudometric over $\RR^d$ satisfying conditions~(\ref{homT}) (ergodicity), (\ref{mesuT}) (annular mesurability)  and~(\ref{annuliT}) (decay of instant one-arms) for some $\eta, R_0>0$. Then,
	there exists $M_0>0$ such that  
	$$ \forall M\geq M_0, \ \exists c_M, \ \forall c\leq c_M, 
	\PP\left[\frac{T(A_M)}{M}\leq c\right]\leq \frac{1}{M^{d-1+\eta/2}}.$$
\end{lem}
\begin{proof}
	By condition~(\ref{annuliT}), there exists $M_0>0$ such that for all $M\geq M_0$,
	$$\PP\big[T(A_M)=0\big]\leq \frac{1}{2M^{d-1+\eta/2}}.$$
	Since
	for a real-valued random variable $X$, the function $x\mapsto \PP(X\leq x)$ is right continuous, we obtain the result.
\end{proof}

We can now prove Proposition~\ref{keyprop}.
\begin{proof}[Proof of Proposition~\ref{keyprop}]
	By condition~(\ref{annuliT}) (decay of instant one-arms) and Lemma~\ref{cross},
	there exists $M_0$ and $\eta>0$ such that 
	\begin{equation}\label{cm}
	 \forall M\geq M_0, \ \exists c_M, \ 
	\PP\left[\frac{T(A_M)}{M}\leq c_M\right]\leq \frac{1}{M^{d-1+\eta}}.
	\end{equation}
	Moreover, by Corollary~\ref{corocomp} there exists $R_0\geq M_0$ and $\epsilon>0$ such that  
	such that for any $R\geq R_0$ and any $\delta>0$,
the implication (\ref{eq}) holds. Let $$\delta:= c_{R_0}$$ 
be defined and given by~(\ref{cm})
and for any integer $k\geq 1$, define:
$$
M_k:=R_0^{(1+\epsilon)^{k-1}}.
$$
Then by an immediate induction and Corollary~\ref{corocomp},
$$ \forall k\geq 1, \ 
	\PP\left[\frac{T(A_{M_k})}{M_k}\leq \delta\prod_{j=0}^{k-1} (1+M^{-\frac{\epsilon}{2}}_j)^{-1}\right]\leq \frac{1}{M_k^{d-1+\eta}}.$$
	Now note that $M_k^{-\frac{\epsilon}{2}} = M^{-(1+\epsilon)^k \frac{\epsilon}{2}},$
so that the product $\prod_{j=0}^{\infty} (1+M_j^{-\frac{\epsilon}{2}})^{-1}$ converges to a constant $\gamma>0$. Hence,
we then obtain 
\begin{equation}\label{grr}
 \forall k\geq 1, 
\PP\left[\frac{T(A_{M_k})}{M_k}\leq \delta \gamma \right]\leq \frac{1}{M_k^{d-1+\eta}},
\end{equation}
which implies the result.
\end{proof}

\paragraph{Critical exponent.} 
We finish this paragraph with the proof of the estimate for the one-arm decay. 
\begin{proof}[Proof of Corollary~\ref{cegT}] If the conclusion does not hold, then $T$ satisfies condition~(\ref{annuliT}) (decay of instant one-arm), so that by Theorem~\ref{positifT}, $\mu>0$, which is a contradiction. 
\end{proof}

\subsection{Vanishing of the time constant}\label{subz}

\paragraph{Instant rescaled annuli crossings.}
We explain why Theorem~\ref{muzeroT} proved in a Boolean setting extends to ours, that is $\mu=0$ if, among others, condition~(\ref{vannuli}) (instant crossings of annuli) is satisfied.
\begin{proof}[Proof of  Theorem~\ref{muzeroT}.]
	Under isotropy and condition~(\ref{lipT}) (Lipschitz), the convergence given by Theorem~\ref{fixeddirT} is uniform. This is proved by~\cite[Theorem 1.1]{gouere2017positivity} for the Boolean setting, but the proof given by \cite[\S B]{gouere2017positivity} only uses the Lipschitz property of $T$ and isotropy. Now, in~\cite[\S 2]{gouere2017positivity}, the authors proved 
	\begin{equation}\label{GT}
	\mu \text{ is a norm } \Rightarrow \PP[ T(A(R,2R))=0]\to_R 0.
	\end{equation}
	In fact they wrote a weaker conclusion, namely $$\mu \text{ is a norm }  \Rightarrow \PP [\Cross_0 (A(R,2R)]\to_R 0,$$ but their proof gives the stronger~(\ref{GT}). Now, under isotropy, $\mu$ is a norm or vanishes, so that the contraposition of~(\ref{GT}) gives the result. 
\end{proof}

\paragraph{White rectangle crossings.}
For random colourings, we needed Proposition~\ref{muzero2}, which provides another criterion given by Russo-Seymour-Welsh conditions. More precisely, that $\mu$ vanishes if $\sigma$ satisfies, among others, condition~(\ref{wRSW}) (weak RSW).
\begin{proof}[Proof of  Proposition~\ref{muzero2}.]
	Let $v\in \mathbb S^{d-1} \subset \RR^d$ and $\epsilon>0$. 
	Fix  
	$$R_{\epsilon}:= [0,1]\times \left[-\frac{1}{2\sqrt{d-1}}\epsilon, \frac{1}{2\sqrt{d-1}}\epsilon\right]^{d-1}$$
	and a rotation $r_v: \RR^d \to \RR^d$ sending $(1,0, \cdots, 0) $ to $v$. 
	Then 
	$$ \forall n\geq 1, \ \Cross_0(n r_v R_{\epsilon}) \subset \{T(0,nv)\leq  n\epsilon \}.$$
	Indeed, the white crossing together with the  the smaller sides of $nr_vR_{\epsilon}$ 
	provide a path of time less than $n\epsilon.$
	We illustrate this in Figure~\ref{rect}.
	\begin{figure}
		\centering
		\includegraphics[width=6.5cm]{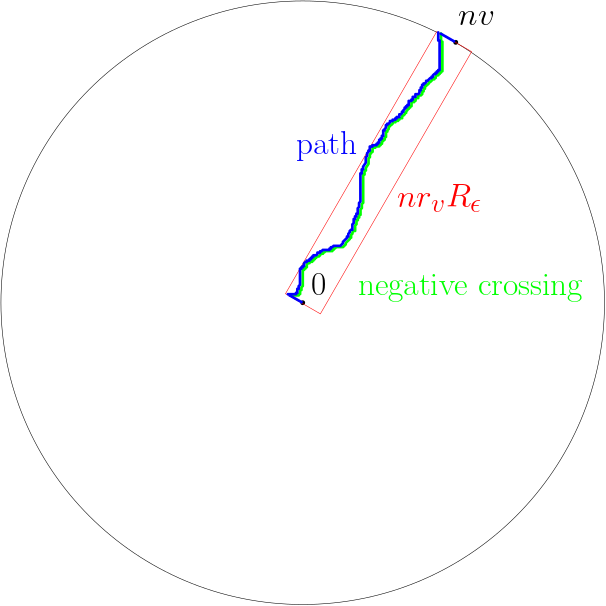}
		\caption{A picture of an event of condition~(\ref{sRSW}) in dimension 2: the narrow rectangle around the line segment $[0,nv]$ is uniformly crossed.}
		\label{rect}
	\end{figure}
	By condition~(\ref{wRSW}) (weak RSW), there exists $c>0$ and a increasing sequence $(s_n)_n$ of integers diverging to infinity such that for any $n$
	\begin{equation}\label{minprob}
	\PP\left[T(0,s_n v)\leq s_n\epsilon\right]\geq \PP\left[\Cross_0(s_n r_vR_{\epsilon})\right] \geq c.
	\end{equation}
	Now, suppose by contradiction that there exists some $\epsilon>0$ with $\mu(v) > \epsilon$. 
	By Corollary~\ref{fixeddirS} we have:
	$$
	\PP\left[\frac{1}n T(0,nv)  \to_{n\to\infty} \mu \right]=1,
	$$
	that is
	$$P:=\PP\left[\bigcap\limits_{m\in \NN}\bigcup\limits_{N\in \NN}\bigcap\limits_{n\geq N}\left|\frac{1}n T(0,nv) -\mu\right|<\frac{1}{m} \right]=1.$$
	The inequality (\ref{minprob}) justifies that for $m$ large enough and for any $n$:
	$$
	\PP\left[ \left|\frac{1}{s_n} T(0,s_nv) -\mu\right|<\frac{1}{m} \right]\leq 1-c.
	$$
	Thus, 
	\begin{eqnarray*}
		P&\leq& \limsup\limits_{m\to\infty}\limsup\limits_{N\to\infty}\PP\left[\bigcap\limits_{n\geq N} \left|\frac{1}{s_n} T(0,s_nv) -\mu\right|<\frac{1}{m} \right]\leq 1-c.
	\end{eqnarray*}
	Hence we get a contradiction, so that $\mu(v)=0$.
\end{proof}
\begin{rem}For this case we only know planar examples. 
\end{rem}

\paragraph{A planar setting.}
We finish this section with the proof of Corollary~\ref{remtassion}. Note that it only concerns vanishing time constants. 
We begin by a result proved by V. Tassion, which holds under pretty weak conditions:
\begin{thm}\cite[\S 2\text { and Remark 3}]{tassion2016crossing}\label{tassioncrossing}
	Let $\sigma: \R^2 \to \{0,1\}$ be a planar colouring. Under the conditions of Corollary~\ref{remtassion}, 
	$\sigma$ satisfies condition~(\ref{wRSW}) (weak RSW).
	\end{thm}
\begin{proof}	[Proof of Corollary~\ref{remtassion}]			
	By Theorem~\ref{tassioncrossing}, $\sigma$ satisfies condition~(\ref{wRSW}) (weak RSW). By Proposition~\ref{muzero2} (RSW implies vanishing $\mu$), this implies that $\mu=0$ and by Corollary~\ref{ballshapeBF} (ball shape theorem), the pseudo-balls defined by $T$ grow faster than the Euclidean ones.
\end{proof}

\subsection{The ball shape theorem}

We set out to prove Theorem~\ref{ballshapeBFT} (ball shape theorem).
Firstly, let us define a particular event:
\begin{itemize} 
	\item Let $T$ be a random pseudometric over $\R^d$ satisfying conditions~(\ref{homT}) (ergodicity) and (\ref{moment}) (finite moment). 
Denote by $E$  the event 
	\begin{equation}\label{conv}
	E:=\left\{\forall b\in \QQ^d, \ \frac{1}n T(0,nb)\to_{n\to+\infty}\mu(b)
	\right\},
	\end{equation}
	where $\mu$ is the time constant defined by Theorem~\ref{fixeddirT}.
	\end{itemize}
Note that by Theorem~\ref{fixeddirT}, $E$ happens almost surely.
For both cases of Theorem~\ref{ballshapeBFT}, $\mu=0$ or $\mu>0$, we will use the same compacity lemma:
\begin{lem}\label{lemlem}
	Let $T$ be a random pseudometric satisfying conditions~(\ref{homT}) (ergodicity),    (\ref{moment}) (finite moment) and~(\ref{lipT}) (Lipschitz). 	
	Assume $E$ is satisfied, 
	and 
	let 
	$(z_n)_n$ be a sequence in $\R^d$ such that	
	$\|z_n\|\to_n +\infty$. 
		Then, there exists a subsequence $(y_n)_n$ of $(z_n)_n$ and $a\in \mathbb S^{d-1}$ such that
			\begin{equation}\label{finish}
			\frac{y_n}{\|y_n\|}\to_n a \text{ and }	\frac{1}{\|y_n\|} T(0,y_n)  \to_n \mu(a).
		\end{equation}
	\end{lem}
\begin{proof} 
	By Lemma~\ref{mulip}, there exists $C_\mu>0$ such that $\mu $ is $C_\mu$-Lipschitz, and 
	by condition~(\ref{lipT}) there exists $C_T>0$ such that $T$ is $C_T$-Lipschitz. 
		By compactness,
	we can assume that there exists a subsequence  $(y_n)_n$ of $(z_n)_n$ and $a\in \mathbb S^{d-1}$, such that
	\begin{equation}\label{compact}
	\frac{y_n}{\|y_n\|}\to_{n\to\infty}a.
	\end{equation}
	Let  $\eta>0$ and $b=b(\eta)\in\QQ^d$ be such that 
	\begin{equation}\label{zed}
	\|a-b\|<\frac{\eta}{9\max(C_T, C_\mu)}.
	\end{equation}
	Let  $N$ be so large that 
	\begin{equation}\label{xy}
	\forall n\geq N,\ \left\|\frac{y_n}{\|y_n\|}-a\right\|<\frac{\eta}{3\max(C_\mu, C_T)}.
	\end{equation}
	Since $\mu$ is $C_\mu$-Lipschitz, (\ref{xy}) implies that 
	\begin{equation}\label{mu}
	\forall n\geq N,  \ 	\big|\mu(y_n)-\|y_n\|\mu(a)\big|<\frac{\eta}{3}\|y_n\|.
	\end{equation}
	Since $E$ given by (\ref{conv}) holds, there exists $N_\eta\geq N$, such that 
	$$
	\forall n\geq N_\eta, \ 		\left|\frac{T(0,\|y_n\|b)}{\|y_n\|}-\mu(b)\right|<\frac{\eta}{9}.
	$$
	Moreover by (\ref{zed}) and since $T$ is $C_T$-Lipschitz, 
	$$\forall n\in \NN, \ \left|\frac{T(0,\|y_n\|a)}{\|y_n\|}-\frac{T(0,\|y_n\|b)}{\|y_n\|}\right|<
	\frac{\eta}{9},
	$$
	so that we have for all $n\geq N_\eta$, using again (\ref{zed}) and that $\mu$ is $C_\mu$-Lipschitz for the last term,
	\begin{eqnarray}\label{y}
	\left|\frac{T(0,\|y_n\|a)}{\|y_n\|}-\mu(a)\right|&\leq &
	\left|\frac{T(0,\|y_n\|a)}{\|y_n\|}-\frac{T(0,\|y_n\|b)}{\|y_n\|}\right| + \left|\frac{T(0,\|y_n\|b)}{\|y_n\|}-\mu(b)\right| \nonumber \\ && + \left|\mu(b)-\mu(a)\right|<\frac{\eta}{3}.
	\end{eqnarray}
	Now, for all $n\geq N_\eta$:
	\begin{eqnarray*}
		\Bigl|T(0,y_n)-\mu(y_n)\Bigr|&\leq& \Bigl\lvert T(0,y_n)-T(0,\|y_n\|a)\Bigr\rvert+\Bigl\lvert T(0,\|y_n\|a)-\mu(\|y_n\|a)\Bigr\rvert\\
		&&+\Bigl\lvert\mu(\|y_n\|a)-\mu(y_n)\Bigr\rvert.
	\end{eqnarray*}
	Since $T$ is $C$-Lipschitz and by (\ref{xy}), for any $n\geq N$ the first term is upper bounded by
	$\frac{\eta}{3}\|y_n\|$.
	By (\ref{y}), for any $n\geq N_\eta$ the second term is bounded by
	$\frac{\eta}{3}\|y_n\|.$
	By (\ref{mu}) the third term is less than $\frac{\eta}{3}\|y_n\|$ for all $n\geq N$.
	We deduce that 
	$$\forall n\geq N_\eta, \ 
	|T(0,y_n)-\mu(y_n)|<\eta\|y_n\|.
	$$
	Hence, we have proved that
	\begin{equation}
	\frac{1}{\|y_n\|} T(0,y_n) -\mu(\frac{y_n}{\|y_n\|}) \to_n 0,
	\end{equation}
	which implies by continuity of $\mu$ and~(\ref{compact}) that 
	\begin{equation}\label{limlim}
	\frac{1}{\|y_n\|} T(0,y_n)\to_n \mu(a) .
	\end{equation}
	
	\end{proof}
\begin{proof}[Proof of Theorem~\ref{ballshapeBFT}]
First, the compact $K$ defined by~(\ref{KK}) is convex. 
	Indeed, since $\mu$ is a semi-norm, for any $x,y\in\RR^d$ and $t\in[0,1]$,
	$$
		\mu\left(tx+(1-t)y\right)\leq\mu\left(tx\right)+\mu\left( (1-t)y\right)=t\mu(x)+(1-t)\mu(y).
	$$
For the rest of the proof, we begin
 with general implications. 
Firstly,
\begin{eqnarray}\label{implik1}
 \forall \epsilon, t>0, \forall x\in \RR^d, \ x\in \frac{1}t B_t\setminus (1+\epsilon)K 
&\Rightarrow &
\mu(tx) -T(0,tx) > \epsilon t\\ \label{implik2}
\text{ and } x\in  (1-\epsilon)K \setminus \frac{1}t B_t
&\Rightarrow&
T(0,tx)-\mu(tx) > \epsilon t.
\end{eqnarray}
Moreover, by Lemma~\ref{mulip}, there exists $C_\mu>0$ such that $\mu $ is $C_\mu$-Lipschitz, 
so that 
\begin{eqnarray}\label{ksmu}
\forall \epsilon, t>0, \forall x\in \RR^d, \ x\in \frac{1}t B_t\setminus (1+\epsilon)K 
&\Rightarrow&
\|x\|\geq \frac{1}{C_\mu}.
\end{eqnarray}
Besides under condition~(\ref{lipT}) there exists $C_T>0$ such that $T$ is $C_T$-Lipschitz, so that
\begin{eqnarray}\label{ksT}
\forall \epsilon, t>0, \forall x\in \RR^d, \ x\in  (1-\epsilon)K \setminus \frac{1}t B_t
&\Rightarrow&
\|x\|\geq \frac{1}{C_T}.
\end{eqnarray}
Lastly, if $\mu$ is a norm, 
\begin{eqnarray}\label{ks3}
\forall \epsilon\in ]0,1[, t>0, \forall x\in \RR^d, \ x\in  (1-\epsilon)K \setminus \frac{1}t B_t
&\Rightarrow&
\|x\|\leq \frac{1-\epsilon}{\mu(\frac{x}{\|x\|})}.
\end{eqnarray}
We now prove the second assertion of Theorem~\ref{ballshapeBFT}.
Let $E$ be the event defined by~(\ref{conv}).
By Theorem~\ref{fixeddirT} (existence of $\mu$), $\PP(E) = 1$. 
For any $\epsilon\in ]0,1[$, 
define
$$ I_\epsilon:= \left\{(1-\epsilon)K\subset\frac{1}tB_t\subset(1+\epsilon)K\text{ for all $t$ large enough}\right\}.$$
 It is enough to prove that 
$$E\subset I_\epsilon.$$ 
Assume on the contrary that there exists $\epsilon>0$ such that $E$ happens but not $I_\epsilon$. By~(\ref{implik1}) and~(\ref{implik2}), it implies that there exists a sequence $(t_n)_n$ of positive reals such that
\begin{equation}\label{tni}
 t_n \to_n +\infty,
\end{equation}
 and a sequence $(x_n)_n \in (\R^d)^\NN$, such that 
\begin{equation}\label{iksen}
\forall n\in \NN, \ x_n \in \left(\frac{B_{t_n}}{t_n}\setminus (1+\epsilon)K \right)
\cup \left((1-\epsilon)K \setminus \frac{B_{t_n}}{t_n}\right) 
\end{equation}
and 
	\begin{equation}\label{debut}
	|T(0,x_nt_n)-\mu(t_nx_n)|\geq \epsilon t_n.
	\end{equation}
For any integer $n$, let
	$$z_n:= t_n x_n.$$ Note that for any $n$, $z_n\not=0$ and
	by~(\ref{tni}), (\ref{iksen}), (\ref{ksmu}) and~(\ref{ksT}), 
	$$\|z_n\|	\to_n +\infty.$$ 
By Lemma~\ref{lemlem}, 
there exists $a\in \R^d$ and a subsequence $(y_n)_n$ of $(z_n)_n$ such that 
\begin{equation}\label{limlimbis}
\frac{y_n}{\|y_n\|} \to_n a \text{ and } \frac{1}{\|y_n\|} T(0,y_n)\to_n \mu(a) .
\end{equation}
Since $\mu$ is a norm,  
 there exists 
	$N'$, such that for $n\geq N'$, 
			\begin{equation}\label{undemi}
			 \frac{1}{\|y_n\|} T(0,y_n)>\frac{1}2\mu(a) .
			 \end{equation}
Fix $n\in \NN$. If $x_n \in \frac{1}{t_n} B_{t_n}\setminus (1+\epsilon) K$, then 
$T(0,y_n) \leq t_n$  so that by~(\ref{undemi}),
$$ \|x_n\| \leq 2\mu(a).$$
If on the contrary $x_n \in (1-\epsilon) K\setminus \frac{1}t B_t,$
by (\ref{ks3})
$$\|x_n\|\leq \frac{1-\epsilon}{\inf_{\mathbb S^{d-1}} \mu}.$$
In all cases,  we see that $(x_n)_n$ is bounded 
so that by~(\ref{finish}) and the continuity of $\mu$ at $a$,
$$	 \frac{1}{t_n} T(0,y_n) -\mu(\frac{y_n}{t_n}) \to_n 0,$$
		which contradicts~(\ref{debut}) and proves the second assertion of Theorem~\ref{ballshapeBFT}.
	
	We prove now the first assertion of the theorem, again by contradiction.
	Assume $\mu=0$, that $E$ is satisfied and that there exists $M>0$ and a sequence $(t_n)_n$ diverging to infinity, such that
	$\forall n,\  \frac{1}{t_n}B(t_n)$ does not contain  $M\mathbb B.$
	Hence, there exists $(x_n) \in (M\mathbb B)^\NN$ such that 
 \begin{equation}\label{bjork}
 \forall n, \ T(0,x_n t_n)> t_n.
 \end{equation}
  As before,
	let $(z_n)_n:= (t_n x)_n$. Then again $\|z_n\|\to_n +\infty$. 
	By Lemma~\ref{lemlem},
	there exists a subsequence $(z_n)_n$ of $(y_n)_n$ such that 
	\begin{equation}\label{limlimbis}
	\frac{1}{\|y_n\|} T(0,y_n)\to_n 0 .
	\end{equation}
	 Because again $(x_n)_n$ is bounded, this implies that 
	$\frac{1}{t_n}T(0,y_n)\to_n 0,$
	which contradicts~(\ref{bjork}).
	\end{proof}

\subsection{Random densities}

\paragraph{Proof of the main results. }
We begin by the proofs of the main three corollaries. 
\begin{proof}[Proof of Corollary~\ref{fixeddirS} (existence of $\mu$)]
	By condition~(\ref{mesu}) (mesurability),
	the random pseudometric $T$ associated with $\sigma$ is well defined and satisfies condition~(\ref{homT}) (ergodicity) because $\sigma$ satisfies~(\ref{hom}) (ergodicity). Condition~(\ref{mesuT}) (annular mesurability) is also fullfilled. Condition~(\ref{momentS}) (finite moment) implies that $T$ satisfies condition~(\ref{moment}). Hence, the first assertion of Theorem~\ref{fixeddirT} can be then applied.
	Finally, $T$ satisfies~(\ref{isoT}) (isotropy) because $\sigma$ satisfies~(\ref{iso}) (isotropy). 
\end{proof}

For densities, we will need the following simple lemma:
\begin{lem}\label{lmesu} Let $\sigma : \R^d \to \R^+$ be a random density satisfying condition~(\ref{mesu}) (mesurability). Then, the associated pseudometric defined by~(\ref{defmetric}) satisfies condition~(\ref{mesuT}) (annular mesurability).
	\end{lem} 
\begin{proof}
	For any finite set of points $x_1,...,x_n$, we denote by $\gamma_{x_1,...,x_n}$ the piecewise affine path starting at $x_1$, ending at $x_n$, going through $x_2,...,x_{n-1}$ in order and following the straight line in between two consecutive $x_i$'s with speed $1$.
	$$
	T(A_{r,R})=\inf\limits_{n\in\NN}\inf\limits_{\substack{x_2,...,x_{n-1}\in\QQ^d\cap A_{r,R}\\ x_1\in\QQ^d\cap \BB_r,x_n\in\QQ^d\setminus \BB_R}}\int_{(\gamma_{x_1,...,x_n})}\sigma.
	$$
	Since any infimum of a sequence of mesurable maps is mesurable, it suffices to show that for a fixed piecewise affine $\gamma: [0,L]\to \R^d$ with $\|\gamma'\|=1$ almost everywhere, the mapping
	$$
	\sigma\mapsto\int_\gamma \sigma
	$$
	is measurable, where the $\Sigma-$algebra for $\sigma$ is the one generated by events depending on a finite number of points.
	For this, we use the approximation by simple functions. Using the condition~(\ref{mesu}) (measurability) and the non-negativity assumption for $\sigma$ we can write
	$$
	\int_{\gamma}\sigma=\sup\limits_{\substack{f\text{ simple function}\\ f\leq\sigma\circ\gamma}}\int_{[0,L]}f.
	$$
	Now, $\sup\int_{[0,L]}f$ is the supremum of terms which depend on the infimums of $\sigma$ on a finite number of segments, which is clearly measurable with respect to the $\Sigma-$algebra associated to $\sigma$.
	In conlusion, $T(A_{r,R})$ is measurable.
	\end{proof}
\begin{proof}[Proof of Corollary~\ref{positif} (positivity of $\mu$)]
By Lemma~\ref{lmesu}, $T$ satisfies condition~(\ref{mesuT}). 
	Repeating the implications of the proof of Corollary~\ref{fixeddirS} and adding the new conditions of $\sigma$ shows that the hypotheses for the $T$ in Theorem~\ref{positifT} are satisfied, so that Corollary~\ref{positif} holds.
\end{proof}

\begin{proof}[Proof of Corollary~\ref{muzero} (vanishing of $\mu$)]
	Since $\sigma$ satisfies condition~(\ref{vannuli}) (white crossings of large annuli), then the associated $T$ satisfies condition~(\ref{vannuliT}) (instant crossings of large annuli). Indeed, 
	$$\{\Cross_0(R,2R) \} \subset \{ T(R,2R) = 0\}.$$
\end{proof}
For the applications, we will need the following general Lemma which ensures that a minimial regularity of the positive region
implies equivalence of the two conditions~(\ref{wannuli}) (decay of white one-arms) and~(\ref{annuliT}) (decay of instant one-arms).
\begin{lem} \label{3a3b}Let $\sigma: \RR^p \to \R_+$ be a random density satisfying condition~(\ref{blackreg}) (positive region regularity). Then it satisfies 
	condition~(\ref{wannuli}) (decay of white one-arms) if and only if it satisfies condition~(\ref{annuliT}) (decay of instant one-arms).
\end{lem}
\begin{proof} Let us prove the stronger assertion that almost surely,
	the events $ \{\Cross_0(A_R)\}$ and $\{T(A_R)=0\}$ happen simultaneously. 
	Indeed, assume that there exists a piecwise $C^1$ path $\gamma$ from $S(0,1)$ to $S(0,R)$, such that $\sigma_{|\gamma}= 0$ almost everywhere. The absence of the first event would imply that there exists a positive  point $x$ in $\gamma$. If $x$ lies the in interior of one of the positive 0-codimension submanifolds given by condition~(\ref{blackreg}), then there exists an open subset of $\gamma$ over which $\sigma >0$, which is a contradiction. If $x $ is not in the interior of the submanifolds, it is on the boundary of one, to which $\gamma$ is necessarily tangent. Then, it can be moved a bit such that $\gamma$ misses the positive region. 
\end{proof}

\paragraph{Critical exponent.} 
\begin{proof}[Proof of Corollary~\ref{ceg}] This is a direct consequence of Corollary~\ref{cegT}. 
	\end{proof}
We now give Theorem~\ref{hugo}, which asserts that the conclusion of Corollary~\ref{ceg} can be obtained in a far shorter and direct way. The proof of this theorem has been provided to us by Hugo Vanneuville.
Moreover, the statement holds with a far milder decorrelation condition:
\begin{enumerate}[resume=condi]
	\item \label{stat} (stationarity) the law of $\sigma$ is invariant under translations.
	\item \label{AI2}  (very weak asymptotic independence)
	$ \Ind^-(Q, 2Q)\to_{Q\to \infty} 0,$
\end{enumerate}
where $\Ind^-$ is defined by~(\ref{ind}). 
\begin{rem}
	\begin{enumerate}
		\item Condition~(\ref{AI2}) is trivially fullfilled by Bernoulli percolation, and is true for the models we handle with in this paper. 
		\item Note also that this condition is far weaker than our condition~(\ref{AI}) (asymptotic independence). For instance, by \cite[Remark 4.3]{muirhead2018sharp}, 
		this (\ref{AI2}) holds for Gaussian fields whose correlation function has a polynomial decay with degree greater than $2$. 
	\end{enumerate}
\end{rem}
\begin{thm}\cite{hugo}\label{hugo} Let $\epsilon, p_0\in \R$
	$(\sigma_p)_{p_0\leq p\leq p_0+\epsilon} : \R^d \to \R^+$ a  family of random densities which is weakly continuous in the parameter $p$, such that   for any $p$, $\sigma_p$ satisfies conditions~(\ref{stat}) and (\ref{AI2}). 
	Assume that for any $p>p_0$, almost surely there exists a continuous path in  $\{\sigma_p=0\}$ from the  origin to infinity. Then, for $\sigma_{p_0},$
	$$ \limsup_{R\to \infty} R^{d-1} \PP[\Cross_0(A_R)]>0.$$
\end{thm} 
Note that this theorem does not demand positive correlation of crossings. 
\begin{proof}[Proof of Theorem~\ref{hugo}]
	
	Firstly, a simple $d$-dimensional packing argument shows that there exists $N$ depending only on $d$ such that for any $R\geq 1$, $A(10R, 20R)$ contains at most $N$ translates $(A_1, \cdots, A_{N})$ of $A(R,2R)$ such that 
	\begin{equation*}
	\{\Cross_0 A(10R,20R)\}\subset \bigcup_{1\leq i<j\leq N
		\atop \dist(A_i, A_j) \geq R  
	}\{\Cross_0(A_i)\}\cap  \{\Cross_0(A_j)\}
	\end{equation*}
	Summing up over the choices of pairs, for any $R\geq 1$, 
	\begin{equation}\label{ret}
	\PP[\Cross_0 A(10R,20R)]\leq N^2 \left(\PP[\Cross_0 A(R,2R)]^2 +\Ind^-(R, 2 R)\right).
	\end{equation}
	Let 
	$$\delta =\frac{1}{2N^2}.$$ 
	By condition~(\ref{AI2}),
	$\exists R_1\geq 1, \forall R\geq R_1, \ \Ind^-(R,2R)\leq \frac{\delta}{2N^2}.$
	Hence by~(\ref{ret}),
	\begin{equation}\label{pastek}
	\forall R\geq R_1, \ \PP[\Cross_0A(R,2 R)]\leq \delta\Rightarrow 
	\PP[\Cross_0A(10R,20R)]\leq \delta.
	\end{equation}
	By an immediate induction, 
	$$\exists R\geq R_1, \PP[\Cross_0A(R,2 R)]\leq \delta \Rightarrow \forall k\geq 1,  \PP[\Cross_0A(10^kR,2 \cdot10^kR)]\leq \delta.$$ 
	Hence, we have proved 
	$$\exists R\geq R_1, \PP[\Cross_0A(R,2 R)]\leq \delta \Rightarrow 
	p=p_0.$$
	
	Now, assume that $p=p_0$ and $\PP[\Cross_0A(R,2 R)]\to_{R\to \infty} 0$. 
	Then there exists $R_0\geq R_1$ such that 
	$$ \PP[\Cross_0A(R_0,2 R_0)]\leq \delta/2.$$
	By continuity in $p$, 
	$ \PP[\Cross_0A(R_0,2 R_0)]\leq \delta$
	holds for any close enough parameter $p>p_0$.
	But the previous argument implies $p=p_0$, a contradiction.

	Finally, a simple $(d-1)$-dimensional packing argument shows that there exists $C>0$ depending only on the dimension $d$ such that for any $R>1$, there exist  $N_R\leq CR^{d-1}$ translates $(a_1,\cdots, a_{N_R})$ of $A_{R/2}$ (see~(\ref{spsh})), such that for any $i, a_i\subset A(R, 2R) $ and the inclusion of events holds:
	$$ \{\Cross_0 A(R,2R)\}\subset \bigcup_1^{N_R}\{ \Cross_0a_i)\},$$
	so that 
	$ \PP[\Cross_0 A(R,2R)]\leq C R^{d-1}\PP[\Cross_0 A_{R/2})].$
	Assume that Corollary~\ref{cebp} does not hold. Then
	$$\PP[\Cross_0 A(R,2R)]\to_{R\to \infty} 0,$$
	which is a contradiction by the latter paragraph.	
\end{proof}

We move now to the applications.

\section{Proof of the applications}\label{secapp}

\subsection{Classical FPP}\label{proofclass}
\begin{proof}[Sketch of proof of Corollary~\ref{corclassical}]
	The existence of $\mu_\nu$ still holds using the classical proof. 
	We can associate to $\sigma_\nu$ the more simple Bernoulli FPP
	$\sigma$ defined by $\sigma(e)= 0$ for an edge $e\in \mathbb E^d$ if $\nu(e)=0$, and $\sigma(e)=1$ in the other case. 
	Condition~(\ref{AI}) (quasi independence) for $\sigma_\nu$ and $\sigma$ is fullfilled because the times are given independently.

	Assume  that $\PP [\nu=0]< p_c(d).$ 
	Then 
	the probability of white one-arm for the Bernoulli case decreases exponentially fast~\cite[Theorem 5.4]{grimmett}, so that 
	condition~(\ref{annuliT}) (decay of instant one-arms) holds for $\sigma_\nu$. 
	Now it happens that our main Theorem~\ref{positifT} and Corollary~\ref{positif} hold in the lattice setting. We did not write down this fact because our main purpose is continuous FPP. 
	By a lattice version of Corollary~\ref{positif}, $\mu_\nu$ is a norm.
\end{proof}
\begin{proof}[Proof of Corollary~\ref{cebp}]
	If $p=p_c$, Theorem~\ref{tp} implies that $\mu=0$. Since
	condition~(\ref{AI}) is satisfied for Bernoulli percolation, 
the hypotheses of Corollary~\ref{ceg} are satisfied, so that Corollary~\ref{cebp} holds. 
	\end{proof}
Theorem~\ref{hugo} provides a shorter and more direct proof of Corollary~\ref{cebp}.

\subsection{Gaussian FPP}

\paragraph{Regularity.}

We begin by recalling two important classical regularity results. 
The first one concerns analytic regularity:
\begin{thm}\cite[\S A.3]{nazarov2}\label{kolmogorov}
	Let $k\in \NN^*$ and $f: \RR^d\to \RR$ be a Gaussian field with covariance $e$, such that
	$e$ can be differentiated at least $k$ times in  $x$ and $k$ times in $y$, and that these derivatives are continuous. Then, almost surely 
	$f$ is $C^{k-1}$. 
\end{thm}
The second one concerns the geometric regularity of the vanishing locus of the field:
\begin{thm}\label{regular}\cite[Lemma 12.11.12]{adler-taylor}
	Let $f: \RR^d\to \RR$ be a Gaussian field, almost surely $C^1$. Then, almost surely $f$ vanishes transversally. In particular, $\{f=0\}$ is empty or has codimension 1.
\end{thm}
For any $p\in \RR$, recall that 
\begin{equation}
\sigma_p = \frac{1}2\left(1+\text{sign} (f+p)\right),
\end{equation} where the sign is considered as $-1$ over $\{f=0\}$. 
By condition~(\ref{sr}) (strong regularity), $f$ is almost surely $C^1$, so that by Lemma~\ref{regular}, the vanishing locus has a vanishing Lebesgue measure, so that the previous choice has no influence on the value of the random pseudometric $T$ for $\sigma_p$ defined by~(\ref{defmetric}).
Theorems~\ref{kolmogorov} and~\ref{regular} have the useful corollary in our FPP situation:
\begin{cor}\label{equiarm}
Let $f: \RR^d \to \RR$ be be a Gaussian field satisfying assumptions (\ref{symmetries}) (symmetries) and (\ref{sr}) (strong regularity). Then for any $p\in \RR$, 
$\sigma_p$ satisfies condition~(\ref{wannuli}) (decay of white one-arms) if and only if it satisfies condition~(\ref{annuliT}) (decay of instant one-arms).
\end{cor}
\begin{proof} By Theorems~\ref{kolmogorov} and then Theorem~\ref{regular}, almost surely the positive region 
	$\{\sigma=1\}$ is a $d$-submanifold with smooth boundary. Consequently, $\sigma$ satisfies condition~(\ref{blackreg}) (positive region regularity).  The result is then a consequence
	of Lemma~\ref{3a3b}.
	\end{proof}
\paragraph{FKG inequality.}

For Gaussian fields, FKG inequality reads:
\begin{thm}\label{pitt}\cite{Pitt1982}, \cite[Lemma A.12]{rivera2019quasi}
Let $f: \RR^d\to \RR$ be a Gaussian field  satisfiying conditions~(\ref{symmetries}) (symmetries), (\ref{sr}) (strong regularity) and (\ref{wp}) (weak positive correlations). Then for any $p\in  \RR$, $\sigma_p$ defined by~(\ref{sip})  satisfies condition (\ref{FKG}) (FKG).
\end{thm}
Otherwise stated, for Gaussian fields with non-negative correlations, positive crossing events are positively correlated. Note that \cite{Pitt1982} was written for Gaussian vectors.

\paragraph{Exponential decay of crossing probabilities.}

When $p=0$, Theorem~\ref{RSW_2} asserts that both probabilities of $\Cross_0(nR)$ and $\Cross_1(nR)$ are uniformly lower bounded by a positive constant when $n$ goes to infinity. When $p\not=0$, this situation changes drastically:
\begin{thm}\label{expdec} \cite[Theorem 9]{rivera2019talagrand}
	Let $f:\R^2 \to \R $ be a planar Gaussian field satisfying assumptions (\ref{symmetries}) (symmetries), (\ref{sr}) (strong regularity), (\ref{wp}) (weak positive correlations) and (\ref{ed}) (strong decay of correlation). For any $p\in \RR$, let $\sigma_p $ be the associated random planar colouring defined by~(\ref{sip}), and  $R\subset \RR^2 $ be a rectangle. 
	Then 
	$$p>0\Leftrightarrow \exists c>0, M_0>0,  \forall M\geq M_0, \ \PP\Big[\Cross_0(MR)\Big]\leq e^{-cM}.$$
\end{thm}
In fact, the assumptions in~\cite{rivera2019talagrand} are far weaker. Former versions of this theorem have been proved before, see \cite[Theorem 1.7]{rivera2017critical} for the Bargmann-Fock field and \cite[Theorem 6.1]{muirhead2018sharp} for fields satisfying the stronger positivity condition~(\ref{sp}). 
We state now a simple corollary of Theorem~\ref{expdec} which will be used for the proof of Theorem~\ref{posBF}, and which relies only on the FKG condition.
\begin{cor}\label{expdecann}Let $f: \R^2 \to \R$ be a planar Gaussian field satisfying assumptions (\ref{symmetries}), (\ref{sr}), (\ref{wp}) and (\ref{ed}), and let $p>0$.
	Then, there exist positive constants $c,M_0$ such that 
	$$
	\forall M\geq M_0, \ 	\PP\left[\Cross_0(A_M)\right]\leq e^{-cM}.
	$$
In particular, $\sigma_p $ satisfies condition~(\ref{wannuli}) (decay of white one-arms).
\end{cor}

\begin{proof}
	\begin{figure}
		\centering
		\includegraphics[width=7cm]{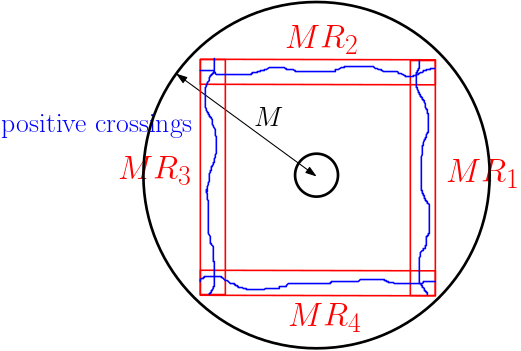}
		\caption{Positive crossings of the four rectangles implies no negative crossing of the annulus}
		\label{ann}
	\end{figure}
	Consider four fixed horizontal or vertical rectangles $(R_i)_{i=1,\cdots,4}$  inside $A_2=A(1,2)$, and such that their open union contains a closed circuit around $A_2$, see Figure~\ref{ann}. Note that
	for all $M\geq 2$ the union of the four copies $MR_1, \cdots MR_4$ 
	lies in $A_M.$
By  Theorem~\ref{expdec}, there exists  $M_0>0$ and $c>0$ such that 
	$$\forall i\in \{1, \cdots, 4\}, \forall M\geq M_0, \ \PP\Bigl[\Cross_1(MR_j)\Bigr]\geq 1-e^{-cM}.
	$$
		Here we used the symmetry of the law under rotation of right angle and by translations.
We also used that a lengthwise positive crossing is the complement event of there being a widthwise negative crossing. By Theorem~\ref{pitt} (FKG) the colouring $\sigma_p$ satisfies condition~(\ref{FKG}) (FKG), so that 
	noting that a positive circuit inside the union of the four rectangles prevents any negative crossing of the annulus,
	\begin{eqnarray*}
		\forall M\geq M_0, \ 		\PP(\Cross_0(A_M)^{\compl })
		\geq \PP\left[\bigcap_{i=1}^4\Cross_1(R_i)\right]\geq (1-e^{-cM})^4.
	\end{eqnarray*}
	Thus there exists $M_1\geq M_0$ and $c'>0$ such that 	
	$$ \forall M\geq M_1, \ 		\PP(\Cross_0(A_M))	\leq e^{-c'M}.$$
\end{proof}
We finish this paragraph with the conclusion:
\begin{cor}\label{conclusion}Let $f: \R^2 \to \R $ be a planar Gaussian field satisfying assumptions (\ref{symmetries}) (symmetries), (\ref{sr}) (strong regularity), (\ref{wp}) (weak positivity) and (\ref{ed}) (strong decay of correlations), and let $p>0$.
	Then $\sigma_p$ satisfies condition~(\ref{annuliT}).
	\end{cor}
\begin{proof}
	By Corollary~\ref{expdecann}, $\sigma_p$ satisfies condition~(\ref{wannuli}). By Corollary~\ref{equiarm}, it thus satisfies condition~(\ref{annuliT}).
	\end{proof}
We emphasize that this corollary is true under much weaker (polynomial with small degree) conditions on the correlation decay, see~\cite{muirhead2018sharp} and~\cite{rivera2019talagrand}. We just did not want to add more conditions.

\paragraph{Asymptotic independence.}

In Bernoulli percolation, the random assignation of a sign to a vertex or an edge is made independently. 
For continuous Gaussian fields, in general the correlation range is infinite, though in our context the correlation converges to zero with the distance. 
In~\cite{beffara2016v}, the authors proved that crossing events in two homothetical large copies of a pair of two  disjoint rectangles were
asymptotically independent if the correlation decay was strong enough. In~\cite{rivera2019quasi}, this result was amended with a simpler and different proof which was in fact close to the one used by V. Piterbarg~\cite{piterbarg}.
In order to ensure that the sign of $\sigma_p$ satisfies the quasi independence condition~(\ref{AI}), we will use another quantitative dependence theorem due to S. Muirhead and H. Vanneuville. Their method differs  from the previous two and has the great advantage for us of holding for general increasing events, not only crossing ones. Note however that it does not need positive correlations of the Gaussian field. 
\begin{thm}\cite[Theorem 4.2]{muirhead2018sharp}\label{muirhead2018sharp}
	Let $f:\R^2 \to \R$ be a planar Gaussian field satisfying conditions (\ref{symmetries})(symmetries), (\ref{sr}) (strong regularity), and (\ref{ed}) (strong decay of correlation) for a certain $\beta>0$. 
	Then, there exists $c, R_0>0$ such that for any $R\geq R_0$, $r\geq 1$ and $t\geq \log R$, for any pair of compact sets $A_1$ and $A_2$ of diameters bounded above by $R$ with $\text{dist}(A_1, A_2)\geq r$, for any events $E_1$, $E_2$ which are both increasing or both decreasing events, and depending only of the field $f$ over $A_1$ and  $A_2$ respectively, we have:
	$$
	\Bigl|\PP(E_1\cap E_2)-\PP(E_1)\PP(E_2)\Bigr|\leq cRtre^{- r^\beta} +ce^{-ct^2}.
	$$
\end{thm}
\begin{proof}
	In~\cite{muirhead2018sharp}, this result was proved with the covariance kernel satisfying a polynomial decay, and not for an exponential one as we need in this paper.
	The first part of the right-hand side was thus polynomial instead of exponential. However, the proof holds in the same way, changing the function $G$ in their Lemma 3.13 to
	$
	G(x)=e^{-\frac12 \|x\|^\beta}. 
	$ The only change in \cite[Theorem 4.2]{muirhead2018sharp} is the $r^{1-\beta}$ term which turns into $re^{-r^\beta}.$ 
	\end{proof}
\begin{cor}\label{coroind} Let $f$ be satisfying the conditions of Theorem~\ref{muirhead2018sharp}. Then 
	for any $0<\beta'<\beta$, and any $p\in \RR$, $\sigma_p$ defined by~(\ref{sip}) satisfies condition~(\ref{AI}) (quasi independence) for $\beta'$. 	
	\end{cor}
\begin{proof}
Firstly, notice that a decreasing event $E$ depending only on the value of $\sigma_p$ over some subset $A\subset \R^d$ is also decreasing for $f$. Hence, by the definition of $\Ind^-$ given by~(\ref{ind}) and Theorem~\ref{muirhead2018sharp} taking $R=S$, $r=Q$ and $t=Q^{\beta}$,
	there exists $Q_0\geq 0$, such that 
	$$ \forall Q\geq (\log S)^{1/\beta}, \ \Ind^-(Q,S) \leq S \exp(-\frac{1}2 Q^\beta).$$
In particular, $\sigma_p$  satisfies condition~(\ref{AI}) for any positive $\beta'<\beta$. 
\end{proof}

\paragraph{Proof of the main Gaussian theorem.}
We can now prove the first main application of this paper.
\begin{proof}[Proof of Theorem~\ref{posBF}.] By condition~(\ref{sr}) (strong regularity)  and Theorem~\ref{kolmogorov} (regularity), almost surely $f$ is continuous, hence locally integrable, so $\sigma_p$ defined by~(\ref{sip}) satisfies condition~(\ref{mesu}) (mesurability). By~\cite[Theorem 6.5.4]{adler}, condition~(\ref{hom}) (ergodicity) holds for centered Gaussian fields which are stationary, which is assumed by condition~(\ref{symmetries}) (symmetries), almost surely continuous, which is true as said before,  and whose correlation function converges to zero at infinity, which is implied by condition~(\ref{wd}) (weak decay). This implies that $\sigma_p$ also satisfies condition~(\ref{hom}).  
By Theorem~\ref{pitt} (FKG) and condition~(\ref{wp}) (weak positivity), $\sigma_p$ satisfies condition~(\ref{FKG}) (FKG). 
	
	If $p=0$, $\sigma_0$ satisfies the symmetry hypotheses of Corollary~\ref{remtassion}, since $f$ satisfies condition~(\ref{symmetries}) (symmetries). Moreover almost surely $f$ is continuous, so that for any horizontal square $S$, $\Cross_1(S)$ occurs if and only if a vertical white crossing of the square does not occur. By symmetries, both have the same probability, which thus is 1/2. Hence, $\sigma_0$ satisfies the conditions of 
	Corollary~\ref{remtassion}, so that $\mu_0=0$ and the first assertion of Theorem~\ref{posBF} is proved for $p=0$. Moreover by Theorem~\ref{tassioncrossing}, $\sigma_0$ also satisfies condition~(\ref{wRSW}) (weak RSW). For $p<0$, since the white crossing probabilities decrease with $p$, $\sigma_p$ satisfies condition~(\ref{wRSW}) (weak RSW), hence the result from Proposition~\ref{muzero2}.
	
	Assume now that $p>0$ and that $f$ satisfies the further condition~(\ref{ed}) (exponential decay of correlations). Then,  Corollary~\ref{conclusion} implies that $\sigma_p$ satisfies condition~(\ref{annuliT}) (decay of instant one-arms). 
		By  Corollary~\ref{coroind}, for any $0<\beta'<\beta$, $\sigma_p$ satisfies condition~(\ref{AI}) (quasi independence) for $\beta'$.
	Corollary~\ref{positif} then implies the second assertion of Theorem~\ref{posBF}. 
	\end{proof}

\subsection{Voronoi FPP}

As in Corollary~\ref{conclusion} for Gaussian fields, we begin with the link between conditions (\ref{wannuli}) (decay of white one-arms) and~ (\ref{annuliT}) (decay of instant one-arms). 
\begin{prop}\label{equiVor}
	Let $\sigma_p:\RR^d \to \{0,1\}$ be Voronoi percolation with parameter $p\in ]0,1[$.
	Then, condition~(\ref{wannuli}) and condition~(\ref{annuliT}) are equivalent.
	Moreover if $p<p_c(d)$, then $\sigma_p$ satisfies condition~(\ref{annuliT}).
\end{prop}
\begin{proof}
The boundaries of the Voronoi cells are defined by inequalities depending through quadratic equations on the points given by the Poisson process, so that condition~(\ref{blackreg}) (positive region regularity) is satisfied.  By Lemma~\ref{3a3b}, condition~(\ref{wannuli}) and condition~(\ref{annuliT}) are equivalent. Now, if $p<p_c(d)$, 
	Theorem~\ref{DRT} (exponential decay of white one-arm) implies that $\sigma_p$ satisfies condition~(\ref{wannuli}) (decay of white one-arms), hence the result. 
	\end{proof}

For condition~(\ref{hom}) (ergodicity) and condition~(\ref{AI}) (asymptotic independence), we will need the following lemmas.
\begin{prop}\cite{tassion2016crossing}\label{bevodec}
Let $p\in]0,1[$ and $\sigma_p $ be the associated Voronoi percolation over $\RR^d$.
Then, there exist constants $c, M_0>0$ such that for all $M\geq M_0$ and $A_1,A_2$ two compact subsets of $\RR^d$, both of diameter less than $M$ and at a distance $\geq M$ from each other, for all events $E_1,E_2$ depending respectively on the colour  over $A_1,A_2$ respectively, we have:
$$
\big|\PP\left[E_1\cap E_2\right]-\PP[E_1]\PP[E_2]\big|\leq e^{-cM^d}.
$$
In particular $\sigma_p$ satisfies condition~(\ref{AI}) (quasi independence).
\end{prop}
The proof of this proposition can be extracted from the proof of Lemma 1.1 of~\cite{tassion2016crossing}. For sake of clarity, we give here a proof of it.
It is a consequence of the following lemma:
\begin{lem}\cite{tassion2016crossing}\label{loin}
	Let $X$ be a Poisson process over $\RR^d$ with intensity $1$, and for $x\in X$, denote by $V_x$ the Voronoi cell based on $x$. Then there exists $c>0$ and $M_0>0$ such that the following holds. For any  open bounded subset $A\subset \RR^d$  with diameter less than $M\geq M_0$, 
	let $E(A,M)$ be the event
	\begin{equation}\label{ki}
	 E(A,M):=\left\{A\subset \bigcup_{x\in X\cap (A+B(0,M))}V_x\right\}.
	 \end{equation}
	Then, 
	$ \PP[E(A,M)]\geq 1-\exp(-cM^d).$
		\end{lem}
	In other terms, with exponentially high probability the Voronoi cells intersecting $A$ do not go too far off of $A$. 
\begin{proof}[Proof of Lemma~\ref{loin}.]
There exists $C>0$, such that for any $M>0$ and $A$ as in the lemma, 
$A$  can be covered by at most $C$ balls of radius  $M$. With probability at 
least $1-C\exp\left(-(\Vol \mathbb B)^d M^d)\right)$, there exists
at least one point of the Poisson process in every ball. Consequently, with the same probability, any point of $A$ 
is $M$-close to a point of the Poisson process. 
\end{proof}	
\begin{proof}[Proof of Proposition~\ref{bevodec}.]
By Lemma~\ref{loin}, with probability at 
	least $1-2e^{-cM^d}$, the event $E(A,M)\cap E(B,M)$ happens, where $E(A,M)$ is defined by~(\ref{ki}). Since the distance beween $A$ and $B$ is larger than $2M$, this implies the result.
\end{proof}	
We could not find in the litterature the proof that the Voronoi percolation is ergodic under the actions of translations, hence the following proposition:
\begin{prop}\label{ergovor}
	For any $p\in \RR$, the translations over $\RR^d$ are ergodic for the Voronoi percolation $\sigma_p$.
\end{prop}
\begin{proof}
	Let $\epsilon>0$ and $A$ an event invariant under the translations. 
	Since $A$ is measurable, there exists a finite number of points $S\subset \RR^d$
and an 	event $A_S$ depending only on the value of $\sigma_p$ on $S$ such that
\begin{equation}\label{rrr}
\PP (A\Delta A_S) \leq \epsilon.
\end{equation}
Let  $c, R_0>\Diam S$ be given by Lemma~\ref{loin} such that 
$$\forall R\geq R_0, \ \PP[E(S,R)]\geq 1- \exp(-cR^d)\leq \epsilon,$$
where $E(S,R)$ is defined by~(\ref{ki}).
Let  
$$v=(4R_0,0, \cdots, 0)\in  \RR^d.$$ 
Then with probability at least $1-\epsilon$, 
$A_S $ is independent of  $\tau_v A_S$,
so that 
\begin{equation}\label{st}
|\PP(A_S\cap \tau_v A_S)-\PP(A_S)^2| \leq \epsilon.
\end{equation}
Since $A$ is invariant under $\tau_v$,
$ \PP(A\cap \tau_v A) = \PP(A).$
Now
	\begin{eqnarray*}
	\PP\Bigl[(A_S\cap \tau_v A_S) \diffsym A \Bigr]&\leq&\PP(A_S \diffsym A)+\PP(\tau_v A_S \diffsym A)\\
& \leq & \PP(A_S \diffsym A)+\PP(A_S \diffsym \tau_{-v} A).
\end{eqnarray*}
	But $\tau_{-v}A=A$. Thus,
	$$\PP\Bigl[(A_S\cap \tau_{v}A_S) \diffsym A \Bigr]\leq2\PP(A_S \diffsym A)\leq 2\epsilon.
	$$
	Therefore,
	$|\PP(A_S\cap \tau_v A_S) -\PP(A)|\leq 2\epsilon.$
	Hence by (\ref{st}) we get
	$$|\PP(A_S)^2-\PP(A)|\leq 3\epsilon.
	$$
	Now using (\ref{rrr}), 
	$$|\PP(A)^2-\PP(A)|\leq 3\epsilon+ |\PP(A_S)^2-\PP(A)^2|\leq 5 \epsilon.$$
	Consequently, $\PP(A)\in \{0,1\}.$
	\end{proof}	
We can now prove the second main application of the general Corollary~\ref{positif}.
\begin{proof}[Proof of Theorem~\ref{bevo} (phase transition for Voronoi FPP)]
The colour of	Voronoi percolation is constant on each tile, and the tiles are semi-algebraic, 
so that $\sigma_p$ satisfies condition~(\ref{mesu}) (mesurability). By Proposition~\ref{ergovor},  $\sigma_p$ satisfies condition~(\ref{hom}) (ergodicity). 
	By Proposition~\ref{bevodec}, $\sigma_p$
	satisfies condition~(\ref{AI}) (quasi independence).

Now, let $p<p_c(d)$.
By Proposition~\ref{equiVor}, $\sigma_p$ satisfies condition~(\ref{annuliT}) (decay of instant one-arms).
Corollary~\ref{positif} then concludes.

For $p>p_c(d)$, by the definition~(\ref{pcdvor}) of $p_c(d)$, almost surely there is an infinite connected component of $\{\sigma_p=0\}$, so that condition~(\ref{vannuli}) (white crossing of large annuli) holds, which implies that $\mu_p=0$ by Proposition~\ref{muzero}.

 If $d=2$, then $p= p_c(2)=1/2$ by Theorem~\ref{BoRi}. By Theorem~\ref{RSWbevo} the colouring $\sigma_p$ satisfies condition~(\ref{RSW}) (RSW), so that  by Proposition~\ref{muzero},
$\mu_{1/2}=0$.
\end{proof}
\subsection{Boolean FPP}

Since this case has been proved in a greater generality, we provide a sketched proof of Corollary~\ref{bouboule}.
\begin{proof}[Sketch of proof of Corollary~\ref{bouboule}]
The model satisfies conditions~(\ref{mesu}) (mesurability) and~(\ref{hom}) (ergodicity).  By Lemma~\ref{loin} and the hypothesis on the exponential tail of the radii, condition~(\ref{AI}) (quasi independence) holds.

Assume $\lambda<\lambda_c$, where $\lambda_c$ is defined by~(\ref{boocrit2}). By Theorem~\ref{boc}, $\sigma_{\nu, \lambda}$ satisfies condition~(\ref{wannuli}) (decay of white one-arms).
By construction, the white region is a locally finite union of non-trivial discs, so that the complementary is defined by quadratic inequalities, hence satisfies condition~(\ref{blackreg}) (positive region regularity). By Lemma~\ref{3a3b}, it hence satisfies condition~(\ref{annuliT}) (decay of instant one-arms).
 Then, Corollary~\ref{positif} implies that $\mu$ is a norm. 
 
 Now if $\lambda>\lambda_c$, condition~(\ref{vannuli}) (white crossings of large annuli) is satisfied since the origin is negatively connected to infinity, so that by Theorem~\ref{muzero},  $\mu=0$. We use again~\cite{gouere2017positivity} for the critical case: Theorem A.1 of said paper implies that in the Boolean case, the subset $\{\lambda>0, \text{ condition~(\ref{vannuli}) is not satisfied}\}$ is open. This implies that for $\lambda=\lambda_c$, $\mu=0$ as well. 
\end{proof}

	\begin{proof}[Proof of Corollary~\ref{ceb}.] 
By  Corollary~\ref{bouboule}, $\mu_{\lambda_c}=0. $
By Remark~\ref{Xor}, the colouring satisfies condition~(\ref{moment}) (finite moment), thus the model satisfies the hypotheses of 
 Corollary~\ref{ceg}, so that its conclusion applies. 
\end{proof}

\subsection{Riemannian FPP}

In this paragraph, we prove Theorem~\ref{smoothmetric0} and its corollaries. 
	\begin{proof}[Proof of Theorem~\ref{smoothmetric0}]
		Let us prove that for any $x$ in $\RR^d$, $\mathbb E T(0,x)$ is finite. For this, 
		note that $$\forall x\in \RR^d, \ T(0,x)\leq \Leng_g([0,x]) = \int_0^1 \|x\|_{g(tx)}dt, $$
		so that by stationarity of $g$, 
		$ \mathbb E T(0,x) \leq \mathbb E \|x\|_{g(0)} .$
		By condition~(\ref{wFM}) (weak finite moment condition), this is finite, so that condition~(\ref{moment}) (finite moment) holds for $T$.  
		Now, condition~(\ref{annuliT}) (decay of instant one-arms) is automatically satisfied, since $T$ is a distance. All the conditions for Theorem~\ref{positifT} are in place, so that it can be applied.
		\end{proof}
\begin{proof}[Proof of Corollary~\ref{smoothmetric}]
	Under the hypotheses of Theorem~\ref{smoothmetricLW}, $g$ has finite correlations, so that condition~(\ref{AI}) (asymptotic independence) is satisfied for the associated pseudometric, thus we can apply Theorem~\ref{smoothmetric0}, which proves the result.
	\end{proof}
\begin{proof}[Proof of Corollary~\ref{smoothmetric2}]
Condition~(\ref{IL}) implies that if $E$ is a decreasing event for the associated pseumetric $T$, then it is also a decreasing event for the function $f$.  Hence, all the conditions are met for Theorem~\ref{muirhead2018sharp}, so that condition~(\ref{AI}) (asymptotic independence) is satisfied for the associated pseudometric, so that we can apply Theorem~\ref{smoothmetric0}.
\end{proof}

\begin{proof}[Proof of Corollary~\ref{smoothmetric3}]
Since $\varphi$ is non-decreasing, the functional $\Leng_g$ is a non-decreasing function in $f$, so that condition~(\ref{IL}) is fullfilled. Moreover,  
let $v\in \RR^d$. Then 
$$\mathbb E \|v\|_{g(0)} = 
\|v\|_{g_0} \mathbb E \varphi(f(0)) = \|v\|_{g_0} \int_\R 
\varphi(u)e^{-\frac{u^2}{2}} \frac{du}{2\pi}$$
which is finite by condition~(\ref{phi1}), so that $g$ satisfies condition~(\ref{wFM}), which implies that Corollary~\ref{smoothmetric2} applies. 
\end{proof}

\subsection{Other models}

\paragraph{Other Gaussian model}
\paragraph{Proof of the other Gaussian theorem.}
We finish this paragraph with the proof of Theorem~\ref{gauss2}.
We will need the classical Borell-TIS inequality:
\begin{prop}\label{Borell}\cite[Theorem 2.1.1]{adler-taylor} Let $A$ be a separable topological space and $f: A\to \R$ be a centered gaussian field over $A$ which is almost
	surely bounded and continuous. Then, $\mathbb E[\sup_{A} f]$ is finite and for all postive $u$,
	$$
	\PP\big[
	\sup_{A} f-\mathbb E (\sup_A f ) > u \big] \leq \exp(-\frac{u^2}{2\sigma_A^2}),$$
	where $\sigma_A^2= \sup_{x\in A} \Var f(x).$
\end{prop}
\begin{cor}\label{corgauss2} Let $f: \RR^p \to \R$ be an ergodic  continuous Gaussian field and $\psi$ an non-decreasing function satisfying~(\ref{psi2}). Let $T$ be the pseudometric defined by~(\ref{gauss2metric}). Then  $T$ satisfies condition~(\ref{moment}) (finite moment).
\end{cor}
\begin{proof}Let $x\in \R^d$ and  $B=B(0,\|x\|)$.
	By Proposition~\ref{Borell}, $\mathbb E (\sup_{B} f)$ is finite. 
	Since
	$$   T(0,x) \leq  \psi (\sup_{B} f)\|x\|, $$
	we obtain that there exists a constant $C_f$, such that
	\begin{eqnarray*}
		\mathbb E T(0,x) &\leq &\sum_{k=0}^\infty 
		\psi(k+1) \PP [ \sup_{B} f\geq k]\\
		& \leq & C_f \int_0^\infty \psi(u) \exp(-\frac{u^2}{2\sigma_{B}^2}) du.
	\end{eqnarray*}
	so that the corollary is true. 
\end{proof}
\begin{proof}[Proof of Theorem~\ref{gauss2}.] For any $p\in \R$, let $T_{p}$ be the pseudometric defined by~(\ref{gauss2metric}) associated with $f+p$.  By the ergodicity of $f$ and thus of $\psi\circ (f+p)$, $T_{p}$ satisfies condition~(\ref{homT}) (ergodicity). By Corollary~\ref{corgauss2} it satisfies condition~(\ref{moment}) (finite moment). Theorem~\ref{fixeddirT} provides the existence of $\mu_p$. Now, all the arguments used in the previous proof of Theorem~\ref{posBF} apply. Indeed, for $p=0$, we still can prove that condition~(\ref{annuliT}) (fast crossings of annuli) is fullfilled, since as the former case, the speed of travelling equals zero over $\{f+p<0\}$. The case $p\leq 0$ is identical. For $p>0$, only condition~(\ref{AI}) is challenging. However, any event $E$ decreasing for $\sigma_p$ is also decreasing for $f$, so that Theorem~\ref{muirhead2018sharp} applies again, and $T_p$ satisfies condition~(\ref{AI}) for any $p$. 
\end{proof}

\paragraph{Ising model. }
\begin{proof}	[Sketch of proof of Corollary~\ref{ising}]
	The Ising model is ergodic and the associated colouring $\sigma_\beta$ is measurable. For $\beta\geq 0$, the model satisfies the FKG inequality, so that we can apply Corollary~\ref{remtassion}, although the model does not have the required symmetries. Indeed, \cite{tassion2016crossing} holds for the symmetries of the triangle lattice.
	For negative $\beta$, this is due to~\cite{beffara2017percolation}, where it is proved that the antiferromagnetic model with high negative temperature satisfies condition~(\ref{sRSW}) (strong RSW), hence condition~(\ref{vannuli}), so that the Proposition~\ref{muzero} concludes. 
\end{proof}

\bibliography{mybib}{}
\bibliographystyle{amsplain}

\end{document}